\documentclass[reqno]{amsart}
\usepackage{amsmath, amsthm, amssymb, amstext}

\usepackage[left=3.5cm,right=3.5cm,top=3cm,bottom=3cm]{geometry}
\usepackage{hyperref,xcolor}% http://ctan.org/pkg/{hyperref,xcolor}
\hypersetup{
	pdfborder={0 0 0},
	colorlinks,
}
\usepackage{enumitem}
\allowdisplaybreaks
\newtheorem{theorem}{Theorem}
\newtheorem{remark}[theorem]{Remark}
\newtheorem{lemma}[theorem]{Lemma}
\newtheorem{proposition}[theorem]{Proposition}

\newtheorem{example}[theorem]{Example}
\DeclareMathOperator*{\Ss}{S}
\DeclareMathOperator*{\dx}{d\textit{x}}
\DeclareMathOperator*{\ds}{d\textit{s}}

\newcommand{\N}{\mathbb{N}}

\newcommand{\R}{\mathbb{R}}

\newcommand{\Lp}[1]{L^{#1}(\Omega)}

\newcommand{\Lprand}[1]{L^{#1}(\partial\Omega)}
\newcommand{\Wp}[1]{W^{1,#1}(\Omega)}

\newcommand{\Wpzero}[1]{W^{1,#1}_0(\Omega)}

\newcommand{\lan}{\langle}
\newcommand{\ran}{\rangle}
\newcommand{\eps}{\varepsilon}
\newcommand{\ph}{\varphi}

\newcommand{\into}{\int_{\Omega}}

\newcommand{\weak}{\rightharpoonup}

\newcommand{\close}{\overline{\Omega}}

\newcommand{\cprime}{$'$}

\renewcommand{\l}{\left}
\renewcommand{\r}{\right}

\newcommand{\WH}{W^{1, \mathcal{H}}(\Omega)}
\newcommand{\WHzero}{W^{1, \mathcal{H}}_0(\Omega)}
\numberwithin{theorem}{section}
\numberwithin{equation}{section}

\def\le{\leqslant}

\def\phi{\varphi}

\def\abs#1{\left|{#1}\right|}
\def\norm#1{\left\|#1\right\|}

\title[Superlinear double phase problems with variable exponents]{Nehari manifold approach for superlinear double phase problems with variable exponents}

\author[\'{A}.\,Crespo-Blanco]{\'{A}ngel Crespo-Blanco}
\address[\'{A}.\,Crespo-Blanco]{Technische Universit\"{a}t Berlin, Institut f\"{u}r Mathematik, Stra\ss e des 17.\,Juni 136, 10623 Berlin, Germany}
\email{crespo@math.tu-berlin.de}

\author[P.\,Winkert]{Patrick Winkert}
\address[P.\,Winkert]{Technische Universit\"{a}t Berlin, Institut f\"{u}r Mathematik, Stra\ss e des 17.\,Juni 136, 10623 Berlin, Germany}
\email{winkert@math.tu-berlin.de}



\begin{document}

\begin{abstract}
	In this paper we consider quasilinear elliptic equations driven by the variable exponent double phase operator with superlinear right-hand sides. Under very general assumptions on the nonlinearity, we prove a multiplicity result for such problems whereby we show the existence of a positive solution, a negative one and a solution with changing sign. The sign-changing solution is obtained via the Nehari manifold approach and, in addition, we can also give information on its nodal domains.
\end{abstract}

\maketitle

%********************************************************************
\section{Introduction}%\label{Introduction}
%********************************************************************

Let  $\Omega \subseteq \R^N$, $N \geq 2$ be a bounded domain with Lipschitz boundary $\partial \Omega$. In this paper we study the following variable exponent double phase problem with homogeneous Dirichlet boundary condition
\begin{equation}
	\label{Eq:Problem}
	\begin{aligned}
		-\operatorname{div} \left(|\nabla u|^{p(x)-2}\nabla u+\mu(x) |\nabla u|^{q(x)-2}\nabla u\right)
		  & = f(x,u) \quad &  & \text{in } \Omega,          \\
		u & = 0            &  & \text{on } \partial\Omega,
	\end{aligned}
\end{equation}
where we assume
\begin{enumerate}[label=\textnormal{(H1)},ref=\textnormal{H1}]
	\item\label{H1}
		$p,q \in C(\close)$ such that $1<p(x)<N$, $p(x)<q(x) \leq q_+ < p^*_-$, where $q_+ = \max_{x\in\close} q(x)$, $p^*_- = \min_{x \in \close} p^*(x)$ and $p^*(x) = Np(x) / (N - p(x))$ for all $x \in \close$, $0\leq \mu(\cdot) \in \Lp{\infty}$ and $p(\cdot)$ satisfies a monotonicity condition, that is, there exists a vector $l \in \R^N \setminus \{ 0 \}$ such that for all $x \in \Omega$ the function
		\begin{equation*}
			g_x (t) = p(x + tl) \quad \text{ with } t \in I_x = \{ t \in \R : x + tl \in \Omega\}
		\end{equation*}
		is monotone.
\end{enumerate}
Furthermore, the nonlinearity $f\colon\Omega\times\R\to\R$ is $(q_+-1)$-superlinear near infinity and $(p( \cdot)-1)$-superlinear near zero with respect to the second variable, see the precise conditions in \eqref{Hf}.

The differential operator in \eqref{Eq:Problem} is the so-called double phase operator with variable exponents and is given by
\begin{align*}
	-\operatorname{div} \l(|\nabla u|^{p(x)-2} \nabla u+ \mu(x) |\nabla u|^{q(x)-2} \nabla u\r)\quad \text{for }u\in \Wp{\mathcal{H}},
\end{align*}
on an appropriate Musielak-Orlicz Sobolev space $\Wp{\mathcal{H}}$.
Special cases of this operator occur when $\inf_{\close} \mu>0$ (the weighted $(q(\cdot),p(\cdot))$-Laplacian) or when $\mu\equiv 0$ (the $p(\cdot)$-Laplace differential operator), which have been studied in the literature before.
The energy functional $I\colon\WHzero \to \R$ related to the variable exponent double phase
operator is given by
\begin{align*}
	I(u)=\into \l( \frac{|\nabla u|^{p(x)}}{p(x)}
	+ \mu(x) \frac{|\nabla u|^{q(x)}}{q(x)}\r)\dx,
\end{align*}
where the integrand $H(x,\xi)=\frac{1}{p(x)}|\xi|^{p(x)}+\frac{\mu(x)}{q(x)}|\xi|^{q(x)}$ for all $(x,\xi) \in \Omega\times \R^N$, according to the nomenclature originally by Marcellini (see his papers \cite{Marcellini-1989b,Marcellini-1991}), has unbalanced growth, that is,
\begin{align*}
	b_1|\xi|^{p(x)} \leq H(x,\xi) \leq b_2 \l(1+|\xi|^{q(x)}\r)
\end{align*}
for a.\,a.\,$x\in\Omega$ and for all $\xi\in\R^N$ with $b_1,b_2>0$.

The most notorious property of the functional $I$ is the nonuniform ellipticity depending on whether we are at the set where the weight function is zero, that is, on the set $\{x\in \Omega\,:\, \mu(x)=0\}$. Indeed, the integrand of $I$ exhibits ellipticity in the gradient of order $q(x)$ on the points $x$ where $\mu(x)\geq \eps > 0$ for any fixed $\eps > 0$ and of order $p(x)$ on the points $x$ where $\mu(x)$ vanishes. So the integrand $H$ switches between two different phases of elliptic behaviours. This is the reason why it is called double phase.

Zhikov \cite{Zhikov-1986} was the first who studied functionals whose integrands change their ellipticity according to a point in order to provide models for strongly anisotropic materials. Functionals like $I$ above, both with constant and variable exponents, have been studied by several authors with respect to regularity of local minimizers. We refer to the works of Baroni-Colombo-Mingione \cite{Baroni-Colombo-Mingione-2015,Baroni-Colombo-Mingione-2016,Baroni-Colombo-Mingione-2018}, Colombo-Mingione \cite{Colombo-Mingione-2015a,Colombo-Mingione-2015b} and the recent results for nonuniformly elliptic variational problems and nonautonomous functionals of Beck-Mingione \cite{Beck-Mingione-2020,Beck-Mingione-2019}, De Filippis-Mingione \cite{De-Filippis-Mingione-2021} and H\"{a}st\"{o}-Ok \cite{Hasto-Ok-2019}.

Double phase differential operators and their corresponding energy functionals given above appear in physical models. For example, in the elasticity theory, the modulating coefficient $\mu(\cdot)$ dictates the geometry of composites made of two different materials with distinct power hardening exponents $q(\cdot)$ and $p(\cdot)$,
see Zhikov \cite{Zhikov-2011}. But also in other mathematical
applications such kind of functional plays an important role, for example, in the study of duality theory
and of the Lavrentiev gap phenomenon, see Papageorgiou-R\u{a}dulescu-Repov\v{s} \cite{Papageorgiou-Radulescu-Repovs-2019b}, Ragusa-Tachikawa \cite{Ragusa-Tachikawa-2020} and
Zhikov \cite{Zhikov-1995,Zhikov-2011}.

Existence results for double phase problems with constant exponents have been shown by several authors within the last decade. The corresponding eigenvalue problem of the double phase operator with Dirichlet boundary condition has been studied by Colasuonno-Squassina \cite{Colasuonno-Squassina-2016} who proved the existence and properties of related variational eigenvalues. Perera-Squassina \cite{Perera-Squassina-2018} showed the existence of a solution by applying Morse theory where they used a cohomological local splitting to get an estimate of the critical groups at zero. Multiplicity results including sign-changing solutions have been obtained by Gasi\'nski-Papa\-georgiou \cite{Gasinski-Papageorgiou-2019},  Liu-Dai \cite{Liu-Dai-2018} and Gasi\'nski-Winkert \cite{Gasinski-Winkert-2021} via the Nehari manifold treatment due to the lack of regularity results for such problems. We also mention the works of Biagi-Esposito-Vecchi \cite{Biagi-Esposito-Vecchi-2021}, Farkas-Winkert \cite{Farkas-Winkert-2021}, Fiscella \cite{Fiscella-2022}, Gasi\'nski-Winkert \cite{Gasinski-Winkert-2020a,Gasinski-Winkert-2020b}, Ge-Pucci \cite{Ge-Pucci-2022}, Stegli\'{n}ski \cite{Steglinski-2022} and Zeng-Bai-Gasi\'nski-Winkert \cite{Zeng-Bai-Gasinski-Winkert-2020a}.

So far, there are only few results involving the variable exponent double phase operator given above. We refer to the recent results of Aberqi-Bennouna-Benslimane-Ragusa \cite{Aberqi-Bennouna-Benslimane-Ragusa-2022} for existence results in complete manifolds, Albalawi-Alharthi-Vetro \cite{Albalawi-Alharthi-Vetro-2022} for convection problems with ($p(\cdot), q(\cdot)$)-Laplace type problems,	Bahrouni-R\u{a}dulescu-Winkert \cite{Bahrouni-Radulescu-Winkert-2020} for double phase problems of Baouendi-Grushin type operator, Crespo-Blanco-Gasi\'nski-Harjulehto-Winkert \cite{Crespo-Blanco-Gasinski-Winkert-2022} for double phase convection problems, Ho-Winkert \cite{Ho-Winkert-2022b} for Kirchhoff problems, Kim-Kim-Oh-Zeng \cite{Kim-Kim-Oh-Zeng-2022}
for concave-convex-type double-phase problems, Leonardi-Papageorgiou \cite{Leonardi-Papageorgiou-2022} for concave-convex problems, Zeng-R\u{a}dulescu-Winkert \cite{Zeng-Radulescu-Winkert-2022} for multivalued problems and Vetro-Winkert \cite{Vetro-Winkert-2023} for multiplicity results under very general growth conditions, see also the references therein. It is also worth mentioning the very recent contribution by Ho-Winkert \cite{Ho-Winkert-2022a} in which they provide an optimal embedding among a certain family of functions and a result about boundedness of the solutions.

The objective of this work is to prove multiplicity results for problem \eqref{Eq:Problem}, where the right-hand side term (possibly nonlinear) is supposed to have a ($q_+ -1$)-superlinear growth at $\pm \infty$. The treatment is inspired by the paper by Gasi\'{n}ski-Winkert \cite{Gasinski-Winkert-2021} on the constant exponents case, and it is interesting to see which are the requirements on the variable exponents to be able to generalize the results. Due to the lack of regularity results for problem \eqref{Eq:Problem}, several tools which are usually applied in the theory of multiplicity results based on the regularity results of Lieberman \cite{Lieberman-1991} and Pucci-Serrin \cite{Pucci-Serrin-2007}, cannot be used in our treatment. Instead we will make use of the mountain pass theorem together with the so-called Nehari manifold, whose definition is motivated by the works of Nehari \cite{Nehari-1960,Nehari-1961}. Further explanations about this method can be found in Section \ref{sign-changing-solutions}.

We also point out that we do not need to suppose conditions like
\begin{align}\label{condition-exponents}
	\frac{q_+}{p_-}<1+\frac{1}{N}
	\quad\text{or}\quad
	\left(\frac{q}{p}\right)^+<1+\frac{1}{N}
\end{align}
as it was used in the Nehari manifold treatments of Gasi\'nski-Papageorgiou \cite{Gasinski-Papageorgiou-2019},  Liu-Dai \cite{Liu-Dai-2018} and Gasi\'nski-Winkert \cite{Gasinski-Winkert-2021} in the constant exponent case. This is due to the fact that the existence of the equivalent norm $\|\nabla \cdot\|_\mathcal{H}$ in $W^{1,\mathcal{H}}_0(\Omega)$ can be proved without supposing \eqref{condition-exponents}, see the paper of Crespo-Blanco-Gasi\'nski-Harjulehto-Winkert \cite[Proposition 2.19]{Crespo-Blanco-Gasinski-Winkert-2022}.

The structure of the paper is the following. In Section \ref{preliminaries} we recall already known properties of the variable exponent spaces $\Lp{p(\cdot)}$ and of the Musielak-Orlicz Sobolev space $\WH$ compatible with the variable exponents double phase operator, among other technical tools that will be used later. In Section \ref{boundedness-solutions} we prove a priori bounds for weak solutions for a class of problems more general than \eqref{Eq:Problem} exploiting the very recent result by Ho-Winkert \cite[Theorem 4.2]{Ho-Winkert-2022a}. In Section \ref{constant-sign-solutions} we introduce the assumptions on the right-hand side $f$ that will be used during the rest of the paper and we prove the existence of a positive and a negative weak solution via the mountain pass theorem applied to functionals truncated at zero. After this, in Section \ref{sign-changing-solutions} we prove the existence of another solution, which turns out to be sign-changing, by solving a minimization problem on a modified version of the Nehari manifold and with the help of the quantitative deformation lemma. Finally, in Section \ref{nodal-domains}, we provide further information on the nodal domains of the sign-changing solution.

%********************************************************************
\section{Preliminaries}\label{preliminaries}
%********************************************************************

In this section we will present the main properties of the Musielak-Orlicz spaces $\Lp{\mathcal{H}}$, $\WH$ and $\WHzero$, together with other relevant results. We denote by $\Lp{r}$ and $L^r(\Omega;\R^N)$ the usual Lebesgue spaces endowed with the norm $\|\cdot\|_r$ for $1\leq r \leq \infty$ and by $\Wp{r}$ and $\Wpzero{r}$ we identify the corresponding Sobolev spaces equipped with the norms $\|\cdot\|_{1,r}$ and $\|\cdot\|_{1,r,0}$, respectively, for $1 \leq r<\infty$.

First, we present the relevant properties for the variable exponent Lebesgue and Sobolev spaces. For any $r\in C(\close)$ we define
\begin{align*}
	r_-=\min_{x\in \close}r(x) \quad\text{and}\quad r_+=\max_{x\in\close} r(x),
\end{align*}
and we also introduce the space
\begin{align*}
	C_+(\close) = \{ r \in C(\close)\,:\, 1 < r_- \}.
\end{align*}
Let $r \in C_+(\close)$ and let $M(\Omega)$ be the set of all measurable functions $u\colon\Omega\to \R$. We denote by $\Lp{r(\cdot)}$ the Lebesgue space with variable exponent, that is
\begin{align*}
	\Lp{r(\cdot)}=\l\{u\in M(\Omega)\,:\, \into |u|^{r(x)}\dx<\infty \r\},
\end{align*}
whose modular is given by
\begin{align*}
	\varrho_{r(\cdot)} (u) = \into |u|^{r(x)}\dx
\end{align*}
and which is endowed with its corresponding Luxemburg norm
\begin{align*}
	\|u\|_{r(\cdot)} =\inf \l \{\lambda>0 \, : \, \into \l(\frac{|u|}{\lambda}\r)^{r(x)}\dx \leq 1 \r\}.
\end{align*}

The properties of these spaces have been extensively studied in the literature, see for example the book by Diening-Harjulehto-H\"{a}st\"{o}-R$\mathring{\text{u}}$\v{z}i\v{c}ka  \cite{Diening-Harjulehto-Hasto-Ruzicka-2011}. The space $\Lp{r(\cdot)}$ endowed with $\| \cdot \|_{r(\cdot)}$ is a separable and reflexive Banach space. The conjugate variable exponent to $r$ is defined by $r' \in C_+(\close)$ such that
\begin{align*}
	\frac{1}{r(x)}+\frac{1}{r'(x)}=1 \quad\text{for all }x\in\close.
\end{align*}
It holds that $\Lp{r(\cdot)}^*=\Lp{r'(\cdot)}$ and a weaker version of the H\"older's inequality holds, given by
\begin{align*}
	\into |uv| \dx \leq \left[\frac{1}{r_-}+\frac{1}{r'_-}\right] \|u\|_{r(\cdot)}\|v\|_{r'(\cdot)} \leq 2 \|u\|_{r(\cdot)}\|v\|_{r'(\cdot)} \quad \text{ for all } u\in \Lp{r(\cdot)}, \; v\in \Lp{r'(\cdot)},
\end{align*}
see Diening-Harjulehto-H\"{a}st\"{o}-R$\mathring{\text{u}}$\v{z}i\v{c}ka  \cite[Lemma 3.2.20]{Diening-Harjulehto-Hasto-Ruzicka-2011}.

Furthermore, if $r_1, r_2\in C_+(\close)$ and $r_1(x) \leq r_2(x)$ for all $x\in \close$, then we have the continuous embedding
\begin{align*}
	\Lp{r_2(\cdot)} \hookrightarrow \Lp{r_1(\cdot)},
\end{align*}
see Diening-Harjulehto-H\"{a}st\"{o}-R$\mathring{\text{u}}$\v{z}i\v{c}ka  \cite[Theorem 3.3.1]{Diening-Harjulehto-Hasto-Ruzicka-2011}

Similarly, one can define the variable exponents spaces with weights: given any $\omega \in \Lp{1}$, $\omega \geq 0$, we define the modular
\begin{align*}
	\varrho_{r(\cdot),\omega} (u) = \into \omega(x) |u|^{r(x)}\dx;
\end{align*}
accordingly, we define the space
\begin{align*}
	L^{r(\cdot)}_\omega(\Omega)=\l\{u \in M(\Omega)\,:\, \into \varrho_{r(\cdot),\omega} (u) \dx<\infty \r\},
\end{align*}
endowed with the corresponding Luxemburg norm
\begin{align*}
	\|u\|_{r(\cdot),\omega} =\inf \l \{\lambda>0 \, : \, \varrho_{r(\cdot),\omega} \left( \frac{u}{\lambda} \right)  \leq 1 \r\}.
\end{align*}

The corresponding variable exponent Sobolev spaces can be defined analogously to the usual way using the variable exponent Lebesgue spaces. A nice introduction to them can be also found in the book by Diening-Harjulehto-H\"{a}st\"{o}-R$\mathring{\text{u}}$\v{z}i\v{c}ka  \cite{Diening-Harjulehto-Hasto-Ruzicka-2011}. For $r \in C_+(\close)$ the variable exponent Sobolev space $\Wp{r(\cdot)}$ is defined by
\begin{align*}
	\Wp{r(\cdot)}=\l\{ u \in \Lp{r(\cdot)} \,:\, |\nabla u| \in \Lp{r(\cdot)}\r\}
\end{align*}
and it is equipped with the norm
\begin{align*}
	\|u\|_{1,r(\cdot)}=\|u\|_{r(\cdot)}+\|\nabla u\|_{r(\cdot)},
\end{align*}
where $\|\nabla u\|_{r(\cdot)}= \|\,|\nabla u|\,\|_{r(\cdot)}$.	Moreover, we define
\begin{align*}
	\Wpzero{r(\cdot)}= \overline{C^\infty_0(\Omega)}^{\|\cdot\|_{1,r(\cdot)}}.
\end{align*}
The spaces $\Wp{r(\cdot)}$ and $\Wpzero{r(\cdot)}$ are both separable and
reflexive Banach spaces, in fact they both possess an equivalent, uniformly convex norm.

A Poincar\'e inequality of the norms holds in the space $\Wpzero{r(\cdot)}$. One way to see this is the paper by Fan-Shen-Zhao \cite[Theorem 1.3]{Fan-Shen-Zhao-2001}, together with the standard way to derive the Poincar\'e inequality from the compact embedding, see for example the paper by Crespo-Blanco-Gasi\'nski-Harjulehto-Winkert \cite[Proposition 2.18 (ii)]{Crespo-Blanco-Gasinski-Winkert-2022}.
\begin{proposition}
	Let $p \in C_+(\close)$, then there exists $c_0>0$ such that
	\begin{align*}
		\|u\|_{r(\cdot)} \leq c_0 \|\nabla u\|_{r(\cdot)}
		\quad\text{for all } u \in \Wpzero{r(\cdot)}
	\end{align*}
\end{proposition}
Thus, we can define the equivalent norm on $\Wpzero{r(\cdot)}$
\begin{align*}
	\|u\|_{1,r(\cdot),0}=\|\nabla u\|_{r(\cdot)}.
\end{align*}

Alternatively, assuming an additional monotonicity condition on $p$,  we also have a Poincaré inequality for the modular, see the paper by Fan-Zhang-Zhao \cite[Theorem 3.3]{Fan-Zhang-Zhao-2005}.
\begin{proposition}\label{Prop:PoincareModular}
	Let $p \in C_+(\close)$ be such that there exists a vector $l \in \R^N \setminus \{ 0 \}$ with the property that for all $x \in \Omega$ the function
	\begin{equation*}
		g_x (t) = p(x + tl) \quad \text{ with } t \in I_x = \{ t \in \R : x + tl \in \Omega\}
	\end{equation*}
	is monotone (either increasing or decreasing). Then there exits a constant $C>0$ such that
	\begin{equation*}
		\varrho_{p(\cdot)} (u) \leq C \varrho_{p(\cdot)} (\nabla u)
		\quad \text{ for all } u \in \Wp{p(\cdot)},
	\end{equation*}
	where $\varrho_{p(\cdot)} (\nabla u) = \varrho_{p(\cdot)} ( | \nabla u | )$.
\end{proposition}

Moreover, we denote by $C^{0, \frac{1}{|\log t|}}(\close)$ the set of all functions $h\colon \close \to \R$ that are log-H\"older continuous, that is, there exists $C>0$ such that
\begin{align*}
	%\label{log_hoelder}
	|h(x)-h(y)| \leq \frac{C}{|\log |x-y||}\quad\text{for all } x,y\in \close \text{ with } |x-y|<\frac{1}{2}.
\end{align*}

For $r \in C_+(\close)$ we introduce the critical Sobolev variable exponents $r^*$ and $r_*$ defined by
\begin{align*}
	r^*(x)=
	\begin{cases}
		\frac{Nr(x)}{N-r(x)} & \text{if }r(x)<N,        \\
		\ell_1(x)            & \text{if } N \leq r(x),
	\end{cases}
	\quad\text{for all }x\in\close
\end{align*}
and
\begin{align*}
	r_*(x)=
	\begin{cases}
		\frac{(N-1)r(x)}{N-r(x)} & \text{if }r(x)<N,        \\
		\ell_2(x)                & \text{if } N \leq r(x),
	\end{cases}
	\quad\text{for all }x\in\close,
\end{align*}
where $\ell_1, \ell_2 \in C(\close)$ are arbitrary functions that satisfy $r(x)<\ell_1(x)$ and $r(x)<\ell_2(x)$ for all $x \in \close$.

We also have Sobolev-type embeddings for the variable exponent Sobolev spaces. The following can be found in Crespo-Blanco-Gasi\'nski-Harjulehto-Winkert \cite[Propositions 2.1 and 2.2]{Crespo-Blanco-Gasinski-Winkert-2022} or Ho-Kim-Winkert-Zhang \cite[Proposition 2.4 and 2.5]{Ho-Kim-Winkert-Zhang-2022}.

\begin{proposition}
	Let $r\in C^{0, \frac{1}{|\log t|}}(\close) \cap C_+(\close)$ and let $s\in C(\close)$ be such that $1\leq  s(x)\leq r^*(x)$ for all $x\in\close$. Then, we have the continuous embedding
	\begin{align*}
		W^{1,r(\cdot)}(\Omega) \hookrightarrow L^{s(\cdot) }(\Omega ).
	\end{align*}
	If $r\in C_+(\close)$, $s\in C(\close)$ and $1\leq s(x)< r^*(x)$ for all
	$x\in\overline{\Omega}$, then this embedding is compact.
\end{proposition}
\begin{proposition}
	Suppose that $r\in C_+(\close)\cap W^{1,\gamma}(\Omega)$ for some $\gamma>N$. Let $s\in C(\close)$ be such that $1\leq  s(x)\leq r_*(x)$ for all $x\in\close$. Then, we have the continuous embedding
	\begin{align*}
		W^{1,r(\cdot)}(\Omega)\hookrightarrow L^{s(\cdot) }(\partial \Omega).
	\end{align*}
	If $r\in C_+(\close)$, $s\in C(\close)$ and $1\leq s(x)< r_*(x)$ for all
	$x\in\overline{\Omega}$, then the embedding is compact.
\end{proposition}

Finally, we recall the relation between the norm and the related modular function, the result is from the paper of Fan-Zhao \cite[Theorems 1.2 and 1.3]{Fan-Zhao-2001}.

\begin{proposition}\label{Prop:PropertiesModularVarExp}
	Let $r\in C_+(\close)$ and $u\in \Lp{r(\cdot)}$.
	\begin{enumerate}
		\item[\textnormal{(i)}]
			If $u\neq 0$, then $\|u\|_{r(\cdot)}=\lambda$ if and only if $\varrho_{r(\cdot)}\l(\frac{u}{\lambda}\r)=1$;
		\item[\textnormal{(ii)}]
			$\|u\|_{r(\cdot)}<1$ (resp. $=1$, $>1$) if and only if $\varrho_{r(\cdot)}(u)<1$ (resp. $=1$, $>1$);
		\item[\textnormal{(iii)}]
			if $\|u\|_{r(\cdot)}<1$, then $\|u\|_{r(\cdot)}^{r_+} \leq \varrho_{r(\cdot)}(u) \leq \|u\|_{r(\cdot)}^{r_-}$;
		\item[\textnormal{(iv)}]
			if $\|u\|_{r(\cdot)}>1$, then $\|u\|_{r(\cdot)}^{r_-} \leq \varrho_{r(\cdot)}(u) \leq \|u\|_{r(\cdot)}^{r_+}$;
		\item[\textnormal{(v)}]
			$\|u\|_{r(\cdot)} \to 0$ if and only if $\varrho_{r(\cdot)}(u)\to 0$;
		\item[\textnormal{(vi)}]
			$\|u\|_{r(\cdot)}\to +\infty$ if and only if $\varrho_{r(\cdot)}(u)\to +\infty$.
	\end{enumerate}
\end{proposition}

Now we introduce the definition and main properties of the Musielak-Orlicz functional space that we will use to study our equations. For the proofs and the rest of the details we refer to the paper of Crespo-Blanco-Gasi\'nski-Harjulehto-Winkert \cite[Section 2]{Crespo-Blanco-Gasinski-Winkert-2022}. To this end, let $\mathcal{H} \colon \Omega \times [0,\infty) \to [0,\infty)$ be defined as
\begin{align*}
	\mathcal{H}(x,t):=t^{p(x)} +\mu(x)t^{q(x)} \quad\text{for all } (x,t)\in \Omega \times [0,\infty),
\end{align*}
where we take similar (but weaker) assumptions as in \eqref{H1}, i.e., we suppose the following:
\begin{enumerate}[label=\textnormal{(H2)},ref=\textnormal{H2}]
	\item\label{H2}
		$p,q\in C_+(\close)$ such that $1<p(x)<N$ and $p(x) < q(x) < p^*(x)$ for all $x\in\close$, and $0 \leq \mu(\cdot) \in \Lp{\infty}$.
\end{enumerate}
It is known that $ \mathcal{H}$ is a locally integrable, generalized $N$-function which satisfies the $\Delta_2$-condition, that is,
\begin{align*}
	\mathcal{H}(x,2t)=(2t)^{p(x)} +\mu(x)(2t)^{q(x)} \leq 2^{q_+} \mathcal{H}(x,t),
\end{align*}
and also that
\begin{align*}
	\inf_{x\in\Omega} \mathcal{H}(x,1)>0.
\end{align*}
The corresponding modular of $\mathcal{H}$ is given by
\begin{align*}
	\rho_{\mathcal{H}}(u) = \into \mathcal{H} (x,|u|)\dx,
\end{align*}
its corresponding Musielak-Orlicz space $\Lp{\mathcal{H}}$ is
\begin{align*}
	L^{\mathcal{H}}(\Omega)=\left \{u \in M(\Omega) \,:\,\rho_{\mathcal{H}}(u) < +\infty \right \},
\end{align*}
and it is endowed with the norm
\begin{align*}
	\|u\|_{\mathcal{H}} = \inf \left \{ \tau >0 : \rho_{\mathcal{H}}\left(\frac{u}{\tau}\right) \leq 1  \right \}.
\end{align*}
This norm is uniformly convex and the space $\Lp{\mathcal{H}}$ is separable and reflexive and satisfies the Radon-Riesz property with respect to the modular $\rho_{\mathcal{H}}$. It has the following relations with the modular and embeddings, see the paper by Crespo-Blanco-Gasi\'nski-Harjulehto-Winkert \cite[Proposition 2.13]{Crespo-Blanco-Gasinski-Winkert-2022}.
\begin{proposition}\label{Prop:PropertiesModularDoublePhase}
	Under hypothesis \eqref{H2}, the following statements hold.
	\begin{enumerate}
		\item[\textnormal{(i)}]
			If $u\neq 0$, then $\|u\|_{\mathcal{H}}=\lambda$ if and only if $ \rho_{\mathcal{H}}(\frac{u}{\lambda})=1$;
		\item[\textnormal{(ii)}]
			$\|u\|_{\mathcal{H}}<1$ (resp.\,$>1$, $=1$) if and only if $ \rho_{\mathcal{H}}(u)<1$ (resp.\,$>1$, $=1$);
		\item[\textnormal{(iii)}]
			if $\|u\|_{\mathcal{H}}<1$, then $\|u\|_{\mathcal{H}}^{q_+}\leqslant \rho_{\mathcal{H}}(u)\leqslant\|u\|_{\mathcal{H}}^{p_-}$;
		\item[\textnormal{(iv)}]
			if $\|u\|_{\mathcal{H}}>1$, then $\|u\|_{\mathcal{H}}^{p_-}\leqslant \rho_{\mathcal{H}}(u)\leqslant\|u\|_{\mathcal{H}}^{q_+}$;
		\item[\textnormal{(v)}]
			$\|u\|_{\mathcal{H}}\to 0$ if and only if $ \rho_{\mathcal{H}}(u)\to 0$;
		\item[\textnormal{(vi)}]
			$\|u\|_{\mathcal{H}}\to +\infty$ if and only if $ \rho_{\mathcal{H}}(u)\to +\infty$;
		\item[\textnormal{(vii)}]
			$\Lp{q(\cdot)} \hookrightarrow
				\Lp{\mathcal{H}} \hookrightarrow
				\Lp{p(\cdot)} \cap L^{q(\cdot)}_\mu(\Omega)$,
	\end{enumerate}
	where $L^{q(\cdot)}_\mu(\Omega)$ is the seminormed space with exponent $q(\cdot)$ and weight $\mu$.
\end{proposition}

In the usual way we can also introduce the corresponding Sobolev spaces
\begin{align*}
	\Wp{\mathcal{H}} := \left \{u \in \Lp{\mathcal{H}} \,:\, |\nabla u| \in \Lp{\mathcal{H}} \right\}
	\quad \text{and} \quad
	\Wpzero{\mathcal{H}} := \overline{C^\infty_0(\Omega)}^{\| \cdot \|_{1,\mathcal{H}}}
\end{align*}
equipped with the norm
\begin{align*}
	\|u\|_{1,\mathcal{H}} = \|u\|_{\mathcal{H}}+\|\nabla u\|_{\mathcal{H}},
\end{align*}
where as above $\|\nabla u\|_\mathcal{H}=\| \, |\nabla u| \,\|_\mathcal{H}$. These spaces are also separable and reflexive, they have an equivalent, uniformly convex norm given by the Luxemburg norm $\| \cdot \|_{\hat{\rho}_\mathcal{H}}$ induced by the modular
\begin{align*}
	\hat{\rho}_\mathcal{H}(u) =\into \l(|\nabla u|^{p(x)}+\mu(x)|\nabla u|^{q(x)}\r)\dx+\into \l(|u|^{p(x)}+\mu(x)|u|^{q(x)}\r)\dx,
\end{align*}
and they also satisfy the Radon-Riesz property with respect to this modular. Furthermore, $\| \cdot \|_{\hat{\rho}_\mathcal{H}}$ and $\hat{\rho}_\mathcal{H}$ satisfy the same relations (i)-(vi) as $\rho_{\mathcal{H}}$ and $\| \cdot \|_{\mathcal{H}}$. Combining the embeddings of variable exponent Lebesgue spaces and $\Lp{\mathcal{H}}$ we have the following Sobolev-type embeddings, see the paper by Crespo-Blanco-Gasi\'nski-Harjulehto-Winkert \cite[Proposition 2.16]{Crespo-Blanco-Gasinski-Winkert-2022}.
\begin{proposition}\label{Prop:EmbeddingsH}
	Under hypothesis \eqref{H2}, the following embeddings hold.
	\begin{enumerate}
		\item[\textnormal{(i)}]
			$\Lp{\mathcal{H}} \hookrightarrow \Lp{r(\cdot)}$, $\WH \hookrightarrow \Wp{r(\cdot)}$ and $\WHzero \hookrightarrow \Wpzero{r(\cdot)}$ are continuous for all $r \in C(\close)$ with $1 \le r(x) \le p(x)$ for all $x \in \close$;
		\item[\textnormal{(ii)}]
			If $p \in C_+(\close) \cap C^{0, \frac{1}{|\log t|}}(\close)$, then $\Wp{\mathcal{H}} \hookrightarrow \Lp{r(\cdot)}$ and $\Wpzero{\mathcal{H}} \hookrightarrow \Lp{r(\cdot)}$ are continuous for $r \in C(\close)$ with $ 1 \leq r(x) \leq p^*(x)$ for all $x\in \close$;
		\item[\textnormal{(iii)}]
			$\Wp{\mathcal{H}} \hookrightarrow \Lp{r(\cdot)}$ and $\Wpzero{\mathcal{H}} \hookrightarrow \Lp{r(\cdot)}$ are compact for $r \in C(\close) $ with $ 1 \leq r(x) < p^*(x)$ for all $x\in \close$;
		\item[\textnormal{(iv)}]
			if $p \in C_+(\close) \cap W^{1,\gamma}(\Omega)$ for some $\gamma>N$, then $\Wp{\mathcal{H}} \hookrightarrow \Lprand{r(\cdot)}$ and $\Wpzero{\mathcal{H}} \hookrightarrow \Lprand{r(\cdot)}$ are continuous for $r \in C(\close)$ with $ 1 \leq r(x) \leq p_*(x)$ for all $x\in \close$;
		\item[\textnormal{(v)}]
			$\Wp{\mathcal{H}} \hookrightarrow \Lprand{r(\cdot)}$ and $\Wpzero{\mathcal{H}} \hookrightarrow \Lprand{r(\cdot)}$ are compact for $r \in C(\close) $ with $ 1 \leq r(x) < p_*(x)$ for all $x\in \close$;
		\item[\textnormal{(vi)}]
			$\Wp{\mathcal{H}} \hookrightarrow \Lp{\mathcal{H}}$ is compact.
	\end{enumerate}
\end{proposition}

For $s \in \R$ let $s^+ = \max \{ s , 0 \}$ and $s^- = \max \{ -s , 0 \}$, hence $s = s^+ - s^-$ and $|s| = s^+ + s^-$. It holds that $\Wp{\mathcal{H}}$ and $\Wpzero{\mathcal{H}}$ are closed under truncations, maxima and minima, see the paper by Crespo-Blanco-Gasi\'nski-Harjulehto-Winkert \cite[Proposition 2.17]{Crespo-Blanco-Gasinski-Winkert-2022}.
\begin{proposition}\label{Prop:ClosedTruncations}
	Under hypothesis \eqref{H2}, the following statements hold.
	\begin{enumerate}
		\item[\textnormal{(i)}]
			If $u \in \WH$, then $u^\pm \in \WH$.
		\item[\textnormal{(ii)}]
			If $u \in \WHzero$, then $u^\pm \in \WHzero$.
	\end{enumerate}
\end{proposition}

On the other hand, in $\Wpzero{\mathcal{H}}$ we have a Poincaré inequality, see the paper by Crespo-Blanco-Gasi\'nski-Harjulehto-Winkert \cite[Proposition 2.18]{Crespo-Blanco-Gasinski-Winkert-2022}.
\begin{proposition} \label{Prop:PoincareIneq}
	Let \eqref{H2} be satisfied. Then there exists a constant $C>0$ such that
	\begin{align*}
		\|u\|_{\mathcal{H}} \leq C\|\nabla u\|_{\mathcal{H}}\quad\text{for all } u \in \Wpzero{\mathcal{H}}.
	\end{align*}
\end{proposition}
This allows us to define the equivalent norm
\begin{align*}
	\|u\|_{1,\mathcal{H},0}=\|\nabla u\|_{\mathcal{H}} \quad\text{for all }u \in \Wpzero{\mathcal{H}},
\end{align*}
which is also uniformly convex and satisfies the Radon-Riesz property with respect to $\rho_{\mathcal{H}} ( \nabla \cdot )$, see the paper by Crespo-Blanco-Gasi\'nski-Harjulehto-Winkert \cite[Proposition 2.19]{Crespo-Blanco-Gasinski-Winkert-2022}.

Lastly, we discuss the properties of the double phase operator defined in the aforementioned space. Let $A\colon \WHzero\to \WHzero^*$ be the nonlinear mapping defined by
\begin{align*}
	\begin{split}
		\langle A(u),v\rangle &=\into \big(|\nabla u|^{p(x)-2}\nabla u+\mu(x)|\nabla u|^{q(x)-2}\nabla u \big)\cdot\nabla v \dx
		\quad \text{for all } u,v\in\WHzero,
	\end{split}
\end{align*}
with $\lan\,\cdot\,,\,\cdot\,\ran$ being the duality pairing between $\WHzero$ and its dual space $\WHzero^*$.  The properties of the operator $A\colon \WHzero\to \WHzero^*$ are summarized in the next proposition, which can be found in the paper by Crespo-Blanco-Gasi\'nski-Harjulehto-Winkert \cite[Proposition 2.18]{Crespo-Blanco-Gasinski-Winkert-2022}.

\begin{proposition}\label{Prop:PropertiesDoublePhaseOp}
	Under hypothesis \eqref{H2}, the operator $A$ is bounded (that is, it maps bounded sets into bounded sets), continuous, strictly monotone (hence maximal monotone), of type $(\Ss_+)$, coercive and a homeomorphism.
\end{proposition}

The last two ingredients of the preliminaries of this paper are the mountain pass theorem and the quantitative deformation lemma. Let $X$ be a Banach space, we say that a functional $\ph : X \to \R$ satisfies the Cerami condition or C-condition if for every sequence $\{u_n\}_{n \in \N} \subseteq X$ such that $\{ \ph(u_n) \}_{n \in \N} \subseteqq \R$ is bounded and it also satisfies
\begin{align*}
	( 1 + \| u_n \| )\ph'(u_n) \to 0 \text{ as } n \to \infty,
\end{align*}
then it contains a strongly convergent subsequence. Furthermore, we say that it satisfies the Cerami condition at the level $c \in \R$ or the C$_c$-condition if it holds for all the sequences such that $\ph (u_n) \to c$ as $n \to \infty$ instead of for all the bounded sequences. The proof of the following mountain pass theorem can be found in the book by Papageorgiou-R\u{a}dulescu-Repov\v{s} \cite[Theorem 5.4.6]{Papageorgiou-Radulescu-Repovs-2019a}.
\begin{theorem}[Mountain pass theorem] \label{Th:MPT}
	Let X be a Banach space, and let $\ph \in C^1(X)$. We assume that $\ph$ satisfies the following properties:
	\begin{itemize}
		\item there exist $u_0, u_1 \in X$ with $\| u_1 - u_0 \| > \delta > 0$, such that
		\begin{align*}
			\max\{\ph(u_0), \ph(u_1)\} &\leq \inf\{\ph(u) : \| u - u_0 \| = \delta \} = m_\delta;
		\end{align*}
		\item $\ph$ satisfies the \textnormal{C}$_c$-condition, where
		\begin{align*}
			c = \inf_{ \gamma \in \Gamma} \max_{ 0 \leq t \leq 1}
			\ph(\gamma (t)) \text{ with } \Gamma &= \{\gamma \in C([0, 1], X) : \gamma(0) = u_0, \gamma(1) = u_1\}.
		\end{align*}
	\end{itemize}
	Then $c \geq m_\delta$ and $c$ is a critical value of $\ph$. Moreover, if $c = m_\delta$, then there exists $u \in \partial B_\delta (u_0)$ such that $\ph'(u) = 0$.
\end{theorem}

Finally, we present here the quantitative deformation lemma. The following version and its proof can be found in the book by Willem \cite[Lemma 2.3]{Willem-1996}.

\begin{lemma}[Quantitative deformation lemma] \label{Le:DeformationLemma}
	Let $X$ be a Banach space, $\ph \in C^1(X;\R)$, $\emptyset \neq S \subseteq X$, $c \in \R$, $\eps,\delta > 0$ such that for all $u \in \ph^{-1}([c - 2\eps, c + 2\eps]) \cap S_{2 \delta}$ it holds that $\| \ph'(u) \|_* \geq 8\eps / \delta$, where $S_{r} = \{ u \in X : d(u,S) = \inf_{u_0 \in S} \| u - u_0 \| < r \}$ for any $r > 0$.
	Then there exists $\eta \in C([0, 1] \times X; X)$ such that
	\begin{enumerate} [label=(\roman*),font=\normalfont]
		\item
			$\eta (t, u) = u$, if $t = 0$ or if $u \notin \ph^{-1}([c - 2\eps, c + 2\eps]) \cap S_{2 \delta}$,
		\item
			$\ph( \eta( 1, u ) ) \leq c - \eps$ for all $u \in \ph^{-1} ( ( - \infty, c + \eps] ) \cap S $,
		\item
			$\eta(t, \cdot )$ is an homeomorphism of $X$ for all $t \in [0,1]$,
		\item
			$\| \eta(t, u) - u \| \leq \delta$ for all $u \in X$ and $t \in [0,1]$,
		\item
			$\ph( \eta( \cdot , u))$ is decreasing for all $u \in X$,
		\item
			$\ph(\eta(t, u)) < c$ for all $u \in \ph^{-1} ( ( - \infty, c] ) \cap S_\delta$ and $t \in (0, 1]$.
	\end{enumerate}
\end{lemma}

%********************************************************************
\section{Boundedness of solutions}\label{boundedness-solutions}
%********************************************************************

In this section we are going to prove that weak solutions of generalized double phase type problems are bounded. This class is much more general than in \eqref{Eq:Problem}, for example, the operator is not restricted to be the variable exponent double phase operator and the right-hand side can depend on the gradient, what is usually called convection term. We assume the following.

\begin{enumerate}[label=\textnormal{(H3)},ref=\textnormal{H3}]
	\item\label{H3}
		Let $\mathcal{A} \colon \Omega \times \R \times \R^N \to \R^N$ and $\mathcal{B} \colon \Omega \times \R \times \R^N \to \R$ be Carath\'eodory functions, i.e. $(t,\xi) \mapsto \mathcal{A}(x,t,\xi)$ is continuous for almost all $x \in \Omega$ and $x \mapsto \mathcal{A}(x,t,\xi)$ is measurable for all $(t,\xi) \in \R \times \R^N$; and analogous conditions for $\mathcal{B}$. Assume that there exist constants $\alpha_1, \alpha_2, \alpha_3, \beta > 0$ and $r \in C_+(\close)$ with $p(x) < r(x) < p^*(x)$ for all $x \in \close$ such that
		\begin{align*}
			\abs{\mathcal{A}(x,t,\xi)}     & \leq \alpha_1 \left[ \abs{t}^{\frac{r(x)}{p'(x)}} + \abs{\xi}^{p(x)-1} + \mu(x) \abs{\xi}^{q(x)-1} + 1 \right] ,       \\
			\mathcal{A}(x,t,\xi) \cdot \xi & \geq \alpha_2 \left[ \abs{\xi}^{p(x)} + \mu(x) \abs{\xi}^{q(x)} \right] - \alpha_3 \left[ \abs{t}^{r(x)} + 1 \right] , \\
			\abs{\mathcal{B}(x,t,\xi)}     & \leq \beta \left[ \abs{\xi}^{\frac{p(x)}{r'(x)}} + \abs{t}^{r(x) - 1} + 1 \right],
		\end{align*}
		for a.\,a.\,$x \in \Omega$ and for all $(t,\xi) \in \R \times \R^N$.
\end{enumerate}

We consider the problem
\begin{equation}
	\label{Eq:ProblemLinfty}
	\begin{aligned}
		-\operatorname{div} \mathcal{A}(x,u,\nabla u)
		  & = \mathcal{B}(x,u,\nabla u) \quad &  & \text{in } \Omega,          \\
		u & = 0                               &  & \text{on } \partial\Omega,
	\end{aligned}
\end{equation}
and we say that $u \in \WHzero$ is a weak solution of \eqref{Eq:ProblemLinfty} if for all $v \in \WHzero$ it holds that
\begin{equation*}
	\into \mathcal{A}(x,u,\nabla u) \cdot \nabla v \dx
	= \into \mathcal{B}(x,u,\nabla u) v \dx.
\end{equation*}

Based on the very recent result of Ho-Winkert \cite[Theorem 4.2]{Ho-Winkert-2022a} we can prove the following result about boundedness of weak solutions of \eqref{Eq:ProblemLinfty}.

\begin{theorem}
	\label{Th:BoundedSolutions}
	Let hypotheses \eqref{H2} and \eqref{H3} be satisfied and let $u \in \WHzero$ be a weak solution of problem \eqref{Eq:ProblemLinfty}. Then, $u \in \Lp{\infty}$ and there exist $C, \tau_1, \tau_2 >0$ independent of $u$ such that
	\begin{equation*}
		\norm{u}_{\infty} \leq C \max \left\lbrace \norm{u}_{r(\cdot)}^{\tau_1} , \norm{u}_{r(\cdot)}^{\tau_2} \right\rbrace.
	\end{equation*}
\end{theorem}

\begin{proof}
	For the proof of this result we follow very closely the proof of \cite[Theorem 4.2]{Ho-Winkert-2022a} with the following changes.

	First, take
	\begin{align*}
		\Psi (x,t) & = t^{r(x)} \quad \text{for all } (x,t) \in \close \times [0,\infty), \\
		Z_n        & = \int_{A_{\kappa_n}} (u - \kappa_n)^{r(x)} \dx,
	\end{align*}
	instead of the definitions given there. Then the Step 1 works exactly the same except for (4.9), which does not hold now.

	Later, take
	\begin{align*}
		T_{n,i} (\alpha) & = \int_{\Omega_i} v_n^\alpha \dx \quad \text{for all } i \in \{1, \ldots , m\},\alpha > 0,          %\\
%		\Psi_\star (x,t) & = t^{r_i^\star + \eps} \quad \text{for all } (x,t) \in \close \times [0,\infty), \star \in \{+,-\}, \\
	\end{align*}
	skip the $\psi_\star$ parts and then use the embedding
	\begin{equation}
		\label{Eq:EmbeddingsLinfty}
		\begin{gathered}
			W^{1,p(\cdot)}_0 (\Omega_i) \hookrightarrow W^{1,(p_i)_-}_0 (\Omega_i) \hookrightarrow L^{r_i^\star + \eps} (\Omega_i), %\\
%			W^{1,p(\cdot)}_0 (\Omega_i) \hookrightarrow L^{\Psi_\star} (\Omega_i),
		\end{gathered}
	\end{equation}
	instead of the ones in the original proof. As a result, one can choose
	\begin{align*}
		R_{n,i} & = \int_{\Omega_i} \left[  \abs{ \nabla v_n}^{p(x)} + \mu(x) \abs{ \nabla v_n}^{q(x)} \right]  \dx \quad \text{for all } i \in \{1, \ldots , m\}, \\
		R_{n}   & = \int_{\Omega} \left[  \abs{ \nabla v_n}^{p(x)} + \mu(x) \abs{ \nabla v_n}^{q(x)} \right]  \dx,
	\end{align*}
	which is why we do not need (4.9) now unlike in the original proof. The rest of Step 2 is identical. Step 3 is exactly the same without any changes.

\end{proof}

\begin{remark}
	Note that our boundedness result holds true with the weaker assumptions on the exponents given in \eqref{H2} instead of the much stronger assumptions needed in \cite[Theorem 4.2]{Ho-Winkert-2022a}. This is because in our simpler setting, we can use the embedding \eqref{Eq:EmbeddingsLinfty} instead of the stronger and sharper embeddings they use there (see (4.18) and (4.19)) and for which they need the mentioned stronger assumptions on the exponents.
\end{remark}

%********************************************************************
\section{Constant sign solutions}\label{constant-sign-solutions}
%********************************************************************

In the present section we are going to prove the existence of constant sign solutions of \eqref{Eq:Problem} by applying the mountain pass theorem to appropriately truncated energy functionals. For this purpose, we assume some of the following hypotheses on the right-hand side function $f$. We include extra assumptions that will not be used in this section, but in future ones for ulterior purposes.

\begin{enumerate}[label=\textnormal{(H$_\textnormal{f}$)},ref=\textnormal{H$_\textnormal{f}$}]
	\item\label{Hf}
		Let $f \colon \Omega \times \R \to \R$ and $F(x,t) = \int_{0}^{t} f(x,s) \ds$.
		\begin{enumerate} [label=\textnormal{(f$_{\arabic*}$)},ref=\textnormal{f$_{\arabic*}$}]
			\item\label{Asf:Caratheodory}
				The function $f$ is Carath\'eodory, i.e. $t \mapsto f(x,t)$ is continuous for almost all $x \in \Omega$ and $x \mapsto f(x,t)$ is measurable for all $t \in \R$.
			\item\label{Asf:WellDef}
				There exists $r \in C_+(\close)$ with $r_+ < p^*_-$ and $C>0$ such that
				\begin{equation*}
					\abs{f(x,t)} \leq C \left( 1 + \abs{t}^{r(x)-1} \right) \quad
					\text{for a.\,a.\,} x \in \Omega \text{ and for all } t \in \R.
				\end{equation*}
			\item\label{Asf:GrowthInfty}
				\abovedisplayskip=0pt\abovedisplayshortskip=0pt \vspace*{-\baselineskip}
				\begin{equation*}
					\lim\limits_{s \to \pm \infty} \frac{F(x,s)}{\abs{s}^{q_+}} = + \infty \quad \text{uniformly for a.\,a.\,} x \in \Omega.
				\end{equation*}
			\item\label{Asf:GrowthZero}
				\begin{equation*}
					\lim\limits_{s \to 0} \frac{F(x,s)}{\abs{s}^{p(x)}} = 0 \quad
					\text{uniformly for a.\,a.\,} x \in \Omega.
				\end{equation*}
			\item\label{Asf:CeramiAssumption}
				There exists $l, \widetilde{l} \in C_+(\close)$ such that $\min \{l_-, \widetilde{l}_- \} \in \left( (r_+ - p_-) \frac{N}{p_-} , r_+ \right) $ and $K > 0$ with
				\begin{align*}
					0 < K \leq \liminf_{s \to + \infty} \frac{f(x,s)s - q_+ 	F(x,s)}{\abs{s}^{l(x)}} \quad
					\text{uniformly for a.\,a.\,} x \in \Omega, \\
					0 < K \leq \liminf_{s \to - \infty} \frac{f(x,s)s - q_+ 	F(x,s)}{\abs{s}^{\widetilde{l}(x)}} \quad
					\text{uniformly for a.\,a.\,} x \in \Omega.
				\end{align*}
			\item\label{Asf:QuotientMono}
				The function $t \mapsto f(x,t)/\abs{t}^{q_+-1}$ is increasing in $(- \infty, 0)$ and in $(0, + \infty)$ for a.\,a.\,$x \in \Omega$,
			\item\label{Asf:IntMono}
%				\begin{equation*}
%					f(x,t)t - q_+F(x,t) \geq 0 \quad \text{for all } t \in \R \text{ and for a.\,a. } x \in \Omega.
%				\end{equation*}
				$f(x,t)t - q_+F(x,t) \geq 0$ for all $t \in \R$ and a.\,a.\,$x \in \Omega$.
		\end{enumerate}
\end{enumerate}

\begin{remark}
	Note that \eqref{Asf:GrowthInfty} and \eqref{Asf:GrowthZero} are weaker than the corresponding assumptions using $f$ instead of $F$. Also note that in \eqref{Asf:GrowthInfty} using as exponent $q_+$ is stronger than using $q(x)$ as exponent, and this is stronger that using $q_-$ as exponent. The analogous statement holds for \eqref{Asf:GrowthZero}.
\end{remark}

\begin{remark}
	The condition on $l_-$ of \eqref{Asf:CeramiAssumption} is always well defined since
	\begin{equation*}
		(r_+ - p_-) \frac{N}{p_-} = r_+ \frac{N}{p_-} - p^*_- \frac{N - p_-}{p_-}
		< r_+ \frac{N}{p_-} - r_+ \frac{N - p_-}{p_-} = r_+.
	\end{equation*}
	Note that this is the precise condition that is needed for the interpolation argument in the Claim of Proposition \ref{Prop:CeramiCondition}. Other works using similar techniques usually only impose that the exponent $l$ is upper bounded by $p_-^*$ (or the corresponding growth exponent of the operator in that work), but a sharper bound with $r_+$ is actually needed. On the other hand, the advantage of the variable exponent setting is that this sharper upper bound is only needed for $l_-$ and not for the whole exponent $l$. Furthermore, one can choose different exponents $l$ and $\widetilde{l}$ for going to $\pm \infty$.
\end{remark}

\begin{example}
	Consider the function
	\begin{equation*}
		f(x,t) =
		\begin{cases}
			|t|^{r_1(x)-2}t[1 + \log(-t)], & \text{if } \phantom{-1}t \leq -1, \\
			|t|^{a(x)-2}t,                 & \text{if } -1 < t < 1,               \\
			|t|^{r_2(x)-2}t[1 + \log(t)],  & \text{if } \phantom{-1} 1 \leq t,
		\end{cases}
	\end{equation*}
	where  $r_1, r_2, a \in C(\close)$, $q_+ \leq a(x)$ and $q_+ \leq r_1(x), r_2(x) < p^*_-$ for all $x \in \close$, and they satisfy
	\begin{equation*}
		\frac{ \max\{(r_1)_+,(r_2)_+\} }{p_-} - \frac{(r_i)_-}{N} < 1, \quad \text{for all } i \in \{ 1, 2\}.
	\end{equation*}
	Then $f$ satisfies all the assumptions above. For \eqref{Asf:WellDef} take $r(x) = \max \{ r_1(x), r_2(x) \} + \eps$ for all $x \in \close$, with $\eps > 0$ small enough so that $r_+ < p_-^*$ and
	\begin{equation*}
		\frac{ r_+ }{p_-} - \frac{(r_i)_-}{N} < 1, \quad \text{for all } i \in \{ 1, 2\}.
	\end{equation*}
	For \eqref{Asf:CeramiAssumption}, take $\widetilde{l}(x) = r_1(x)$, $l(x) = r_2(x)$ for all $x \in \close$. This is the reason for the assumption on $(r_1)_\pm$ and $(r_2)_\pm$. Observe that if we take $r_1 = r_2 = r$ constant, the condition is equivalent to $r < p^*_-$, hence redundant in that case.
\end{example}

The following properties follow from these assumptions.

\begin{lemma}
	\label{Le:Propf}
	Let $f \colon \Omega \times \R \to \R$, the following implications hold.
	\begin{enumerate}[label=(\roman*),font=\normalfont]
		\item\label{Propf:ZeroAtOrigin}
			If $f$ fulfills \eqref{Asf:Caratheodory} and \eqref{Asf:GrowthZero}, then $f(x,0) = 0$ for a.\,a.\,$x \in \Omega$.
		\item\label{Propf:q<r}
			If $f$ fulfills \eqref{Asf:WellDef} and  \eqref{Asf:GrowthInfty}, then $q_+ < r_-$.
			\footnote{Remember that by assumption $q_+<p^*_-$ for all $x \in \close$, thus such $r \in C_+(\close)$ satisfying $q_+<r(x)<p^*_-$ for all $x \in \close$ can always exist.}
		\item\label{Propf:Boundedbelow}
			If $f$ fulfills \eqref{Asf:WellDef} and \eqref{Asf:GrowthInfty}, then there exist some $M>0$ such that
			\begin{equation*}
				F(x,t) > - M \quad \text{for a.\,a.\,} x \in \Omega \text{ and for all } t \in \R.
			\end{equation*}
		\item\label{Propf:EpsilonUpperBoundF}
			If $f$ fulfills \eqref{Asf:WellDef} and \eqref{Asf:GrowthZero}, then for each $\eps > 0$ there exists $C_\eps > 0$ such that
			\begin{equation*}
				\abs{F(x,t)} \leq \frac{\eps}{p(x)} \abs{t}^{p(x)} + C_\eps \abs{t}^{r(x)} \quad \text{for a.\,a.\,} x \in \Omega \text{ and for all } t \in \R.
			\end{equation*}
		\item\label{Propf:EpsilonLowerBoundF}
			If $f$ fulfills \eqref{Asf:WellDef} and \eqref{Asf:GrowthInfty}, then for each $\eps > 0$ there exists $C_\eps > 0$ such that
			\begin{equation*}
				F(x,t) \geq \frac{\eps}{q_+} \abs{t}^{q_+} - C_\eps  \quad \text{for a.\,a.\,} x \in \Omega \text{ and for all } t \in \R.
			\end{equation*}
		\item
			\label{Propf:ConvergenceFterm}
			If $f$ fulfills \eqref{Asf:Caratheodory} and \eqref{Asf:WellDef}, then the functional $I_f \colon \WHzero \to \R$ given by
			\begin{equation*}
				I_f(u) = \into F(x,u) \dx
			\end{equation*}
			and its derivative $I_f ' \colon \WHzero \to \WHzero^*$, given by
			\begin{equation*}
				\left \lan I_f ' (u) , v \right \ran = \into f(x,u) v \dx,
			\end{equation*}
			are strongly continuous, i.e.\,$u_n \weak u$ in $\WHzero$ implies $I_f(u_n) \to I_f(u)$ in $\R$ and $I_f '(u_n) \to I_f ' (u)$ in $\WHzero^*$.
	\end{enumerate}
\end{lemma}

We provide no proof since these statements are either elementary or widely known.

We say that $u \in \WHzero$ is a weak solution of \eqref{Eq:Problem} if for all $v \in \WHzero$ it holds that
\begin{align*}
	\into \left( \abs{ \nabla u}^{p (x) -2} \nabla u + \mu (x) \abs{ \nabla u}^{q (x) -2} \nabla u \right) \cdot \nabla v \dx
	= \into f(x,u) v \dx.
\end{align*}

Another way to consider these weak solutions is to see them as critical points of the energy functional $\ph \colon \WHzero \to \R$ associated to the problem \eqref{Eq:Problem}, which is defined by
\begin{equation*}
	\ph(u)
	= \into \left(  \frac{\abs{\nabla u}^{p(x)}}{p(x)} + \mu(x) \frac{\abs{ \nabla u}^{q(x)}}{q(x)} \right)  \dx  \;
	- \into F(x,u) \dx.
\end{equation*}

In order to look at constant sign solutions, we truncate the functional by zero from above and below respectively, producing the functionals
\begin{equation*}
	\ph_\pm(u)
	= \into \left(  \frac{\abs{\nabla u}^{p(x)}}{p(x)} + \mu(x) \frac{\abs{ \nabla u}^{q(x)}}{q(x)} \right)  \dx  \;
	- \into F_\pm(x, u) \dx,
\end{equation*}
where $F_\pm(x,u) = \int_0^u f(x, \pm s ^\pm) \ds $. Note that by Lemma \ref{Le:Propf} \ref{Propf:ZeroAtOrigin}, this is equivalent to
\begin{equation*}
	\ph_\pm(u)
	= \into \left(  \frac{\abs{\nabla u}^{p(x)}}{p(x)} + \mu(x) \frac{\abs{ \nabla u}^{q(x)}}{q(x)} \right)  \dx  \;
	- \into F(x, \pm u^\pm) \dx.
\end{equation*}

The plan for the rest of the section is to verify the assumptions of the mountain pass theorem (Theorem \ref{Th:MPT}). We start by checking the so-called ``mountain pass geometry''.

\begin{proposition}\label{Prop:PhLowerBound}
	Let \eqref{H1} be satisfied and $f$ fulfill \eqref{Asf:Caratheodory}, \eqref{Asf:WellDef} and \eqref{Asf:GrowthZero}. Then there exist constants $C_1,C_2,C_3 >0$ such that
	\begin{equation*}
		\ph(u), \ph_\pm (u) \geq
		\begin{cases}
			C_1 \| u \|_{1, \mathcal{H}, 0}^{q_+} - C_2 \| u \|_{1, \mathcal{H}, 0}^{r_-}, & \text{if } \| u\|_{1, \mathcal{H}, 0} \leq \min\{1,C_3\},  \\
			C_1 \| u \|_{1, \mathcal{H}, 0}^{p_-} - C_2 \| u \|_{1, \mathcal{H}, 0}^{r_+}. & \text{if } \| u\|_{1, \mathcal{H}, 0} \geq \max\{1,C_3\}.
		\end{cases}
	\end{equation*}
\end{proposition}

\begin{proof}
	Here only the case for $\ph$ will be shown. For the other two cases note in the first inequality that $\varrho_{p(\cdot)} ( \pm u^\pm ) \leq \varrho_{p(\cdot)} ( u )$ and $\varrho_{r(\cdot)} ( \pm u^\pm ) \leq \varrho_{r(\cdot)} ( u )$ and the rest of the proof is identical.

	Let $u \in \WHzero$. By Lemma \ref{Le:Propf} \ref{Propf:EpsilonUpperBoundF}, Poincar\'e inequality for the modular in $W_0^{1,p(\cdot)} (\Omega)$ with constant $C_{p(\cdot)}$, see Proposition \ref{Prop:PoincareModular}, the embedding of $\WHzero \hookrightarrow \Lp{r(\cdot)}$ with constant $C_\mathcal{H}$, see Proposition \ref{Prop:EmbeddingsH} (iii), and Proposition \ref{Prop:PropertiesModularVarExp} (iii) and (iv), it follows
	\begin{align*}
		\ph (u)
		 & \geq \frac{1}{p_+} \varrho_{p(\cdot)} ( \nabla u ) + \frac{1}{q_+} \varrho_{q(\cdot),\mu} ( \nabla u )
		- \frac{\eps}{p_-} \varrho_{p(\cdot)} ( u ) - C_\eps \varrho_{r(\cdot)} ( u )                                                                                                                                      \\
		 & \geq \left( \frac{1}{p_+} - \frac{C_{p(\cdot)} \eps}{p_-} \right) \varrho_{p(\cdot)} ( \nabla u ) +\frac{1}{q_+} \varrho_{q(\cdot),\mu} ( \nabla u )
		- C_\eps \max_{k \in \{ r_+, r_- \} } \{ C_\mathcal{H} ^k \| u \|_{1,\mathcal{H},0}^k \}                                                                                                                           \\
		 & \geq \min \left\{ \frac{1}{p_+} - \frac{C_{p(\cdot)} \eps}{p_-} , \frac{1}{q_+} \right\} \rho_\mathcal{H} (\nabla u) - C_\eps \max_{k \in \{ r_+, r_- \} } \{ C_\mathcal{H} ^k \| u \|_{1,\mathcal{H},0}^k \}.
	\end{align*}
	By picking $0 < \eps < \frac{p_-(q_+-p_+)}{C_{p(\cdot)} q_+ p_+}$ and $C_3 = 1/C_{\mathcal{H}}$,
	and by applying Proposition \ref{Prop:PropertiesModularDoublePhase} (iii) and (iv) the result follows with
	\begin{equation*}
		C_1 = \frac{1}{q_+}, \quad C_2 = C_\eps C_\mathcal{H}^{r_-} \quad \text{ for } \| u\|_{1, \mathcal{H}, 0} \leq C_3, \quad  C_2 = C_\eps C_\mathcal{H}^{r_+} \quad \text{ for } \| u\|_{1, \mathcal{H}, 0} > C_3.
	\end{equation*}
\end{proof}

The next result follows directly from Proposition \ref{Prop:PhLowerBound}.

\begin{proposition}\label{Prop:RingOfMountains}
	 Let \eqref{H1} be satisfied and $f$ fulfill \eqref{Asf:Caratheodory}, \eqref{Asf:WellDef} with $q_+<r_-$ and \eqref{Asf:GrowthZero}. Then there exist $\delta >0$ such that
	\begin{equation*}
		\inf_{\norm{u}_{1,\mathcal{H},0} = \delta} \ph(u) >0
		\quad \text{ and } \quad
		\inf_{\norm{u}_{1,\mathcal{H},0} = \delta} \ph_\pm(u) > 0.
	\end{equation*}
	Moreover, there exists $\delta' > 0 $ such that $\ph(u) > 0$ for $0 < \norm{u}_{1,\mathcal{H},0} < \delta'$.
\end{proposition}

\begin{proposition}
	\label{Prop:PointBeyondMountains}
	Let \eqref{H1} be satisfied and $f$ fulfill \eqref{Asf:Caratheodory}, \eqref{Asf:WellDef} and \eqref{Asf:GrowthInfty}. Let $0 \neq u \in \WHzero$, then $\ph(tu) \xrightarrow{t \to \pm \infty} - \infty$. Furthermore, if $u \geq 0$ a.\,e.\,in $\Omega$, $\ph_\pm (tu) \xrightarrow{t \to \pm \infty} - \infty$.
\end{proposition}

\begin{proof}
	First we prove the assertion for $\ph$. Fix $0 \neq u \in \WHzero$, $\abs{t} \geq 1$ and $\eps \geq 1$, by Lemma \ref{Le:Propf} \ref{Propf:EpsilonLowerBoundF} we can derive
	(note that $ \norm{u}_{q_+} < \infty$ by the embedding $\WHzero \hookrightarrow \Lp{q_+}$ from Proposition \ref{Prop:EmbeddingsH} (iii) and Lemma \ref{Le:Propf} \ref{Propf:q<r})
	\begin{align*}
		\ph(tu) & \leq \frac{\abs{t}^{p_+}}{p_-} \varrho_{p(\cdot)} ( \nabla u ) + \frac{\abs{t}^{q_+}}{q_-} \varrho_{q(\cdot),\mu} ( \nabla u )
		- \frac{\eps \abs{t}^{q_+}}{q_+} \norm{u}_{q_+}^{q_+} + C_\eps |\Omega|\\
		& = \frac{\abs{t}^{p_+}}{p_-} \varrho_{p(\cdot)} ( \nabla u ) + \abs{t}^{q_+} \left( \frac{ \varrho_{q(\cdot),\mu} ( \nabla u ) }{q_-}
		- \eps \frac{ \norm{u}_{q_+}^{q_+} }{q_+} \right) + C_\eps |\Omega|.
	\end{align*}
	Choosing $ \eps $ big enough such that the second term is negative it follows that $\ph(tu) \xrightarrow{t \to \pm \infty} - \infty$.

	For the cases for $\ph_\pm$ one only has to note that if $u \geq 0$ a.\,e.\,in $\Omega$, $\ph_\pm (tu) = \ph(tu)$ for $\pm t > 0$.
\end{proof}

Finally, in order to apply the mountain pass theorem, it is only left to see that the necessary compactness condition is satisfied.

\begin{proposition}\label{Prop:CeramiCondition}
	Let \eqref{H1} be satisfied and $f$ fulfill \eqref{Asf:Caratheodory}, \eqref{Asf:WellDef}, \eqref{Asf:GrowthInfty} and \eqref{Asf:CeramiAssumption}. Then the functionals $\ph_\pm$ satisfy the (C)-condition.
\end{proposition}

\begin{proof}
	We write the proof for $\ph_+$, the case for $\ph_-$ is very similar.

	Let $\{u_n\}_{n \in \N} \subseteq \WHzero$ be a sequence such that
	\begin{align}
		\abs{\ph_+(u_n)} &\leq M_1 \quad \text{for } M_1 > 0 \text{ and for all } n \in \N, \label{Eq:C-bound}\\
		\label{Eq:C-conv}
		(1 + \norm{u_n}_{1,\mathcal{H},0}) \ph_+ ' (u_n) &\to 0 \quad \text{in } \WHzero^*.
	\end{align}
	In the following, recall that for $v \in \WHzero$, there holds that $v^+ \in \WHzero$ by Proposition \ref{Prop:ClosedTruncations}. By \eqref{Eq:C-conv}, for all $v \in \WHzero$ and for a sequence $\eps_n \to 0$ as $n \to \infty$ it holds
	\begin{align}\label{Eq:C-convTestedV}
		\begin{split}
			&\abs{\into \left( \abs{ \nabla u_n}^{p(x)-2} \nabla u_n + \mu (x) \abs{ \nabla u_n}^{q(x)-2} \nabla u_n \right) \cdot \nabla v \dx
			- \into f(x, u_n^+) v \dx}\\
			&\leq \frac{\eps_n \norm{v}_{1,\mathcal{H},0}}{1 + \norm{u_n}_{1,\mathcal{H},0}}.
		\end{split}
	\end{align}
	So for the case $v = - u_n ^- \in \WHzero$, as the supports of $ u_n^+$ and $- u_n^-$ do not overlap, it follows
	\begin{equation*}
		\rho_{\mathcal{H}} \left( - \nabla u_n ^- \right)
		= \into \left( \abs{ \nabla u_n^-}^{p(x)} + \mu (x) \abs{ \nabla u_n^-}^{q(x)} \right) \dx
		\leq \eps_n \quad \text{for all } n \in \N,
	\end{equation*}
	hence by Proposition \ref{Prop:PropertiesModularDoublePhase} (v)
	\begin{equation}\label{Eq:NegPartConvegence}
		- u_n ^- \to 0 \quad \text{in } \WHzero.
	\end{equation}

	{\bf Claim:} The sequence $\{ u_n ^+ \}_{n\in\N} \subseteq \WHzero$ is bounded.

	First, from \eqref{Eq:C-bound} and \eqref{Eq:NegPartConvegence} it follows
	\begin{equation*}
		\rho_{\mathcal{H}} ( \nabla u_n ^+ )  - \into q_+ F (x,u_n^+) \dx \leq M_2 \quad \text{for all } n \in \N,
	\end{equation*}
	and choosing $v = u_n^+ \in \WHzero$ in \eqref{Eq:C-convTestedV} gives
	\begin{equation*}
		- \rho_{\mathcal{H}} ( \nabla u_n ^+ )
		+ \into f(x,u_n^+) u_n^+ \dx \leq \eps_n \quad \text{for all } n \in \N.
	\end{equation*}
	Combining these two equations yields
	\begin{equation*}
		\into \left( f(x,u_n^+) u_n^+ - q_+ F (x,u_n^+) \right)
		\dx \leq M_3 \quad \text{for all } n \in \N.
	\end{equation*}
	Now, without loss of generality, assume that $l_- \leq \widetilde{l}_-$. By \eqref{Asf:CeramiAssumption}, there exist $\tilde{K}, K_0 > 0$ such that
	\begin{equation*}
		\tilde{K} |s|^{l_-} - K_0
		\leq f(x,s)s - q_+ F(x,s)
		\quad \text{for all } s \in \R \text{ and for a.\,a.\,} x \in \Omega,
	\end{equation*}
	which together with the previous equation results in
	\begin{equation}\label{Eq:BoundedCerami}
		\norm{u_n^+}_{l_-} \leq M_4 \quad \text{for all } n \in \N.
	\end{equation}

	By \eqref{Asf:WellDef} and \eqref{Asf:CeramiAssumption} we know that $l_- < r_+ < p^*_-$, hence there exists $t \in (0,1)$ such that
	\begin{equation*}
		\frac{1}{r_+} = \frac{t}{p^*_-} + \frac{1-t}{l_-}.
	\end{equation*}
	The interpolation inequality (see, for example, Papageorgiou-Winkert \cite[Proposition 2.3.17]{Papageorgiou-Winkert-2018}) and \eqref{Eq:BoundedCerami} yield
	\begin{equation}
		\label{Eq:NormChangeCerami}
		\norm{u^+_n}_{r_+}^{r_+}
		\leq \left( \norm{u^+_n}^{t}_{p^*_-} \norm{u^+_n}^{1-t}_{l_-} \right)^{r_+}
		\leq M_4^{(1-t)r_+} \norm{u^+_n}^{t r_+}_{p^*_-} \quad \text{for all } n \in \N.
	\end{equation}
	We may assume that $\norm{u_n^+}_{1,\mathcal{H},0} \geq 1$ for all $n \in \N$. Testing again in \eqref{Eq:C-convTestedV} with $v=u^+_n$ together with Proposition \ref{Prop:PropertiesModularDoublePhase} (iv) and \eqref{Asf:WellDef} one has
	\begin{equation*}
		\norm{u^+_n}^{p_-}_{1,\mathcal{H},0}
		\leq \rho_{\mathcal{H}} ( \nabla u^+_n )
		\leq M_5 (1 + \norm{u^+_n}^{r_+}_{r_+}) \quad \text{for all } n \in \N,
	\end{equation*}
	and by \eqref{Eq:NormChangeCerami} and the embedding $\WHzero \hookrightarrow \Wpzero{p_-} \hookrightarrow \Lp{p^*_-}$, which is given by Proposition \ref{Prop:EmbeddingsH} (i) and classical Sobolev embeddings, we get
	\begin{equation*}
		\norm{u^+_n}^{p_-}_{1,\mathcal{H},0}
		\leq M_6 \left( 1 + \norm{u^+_n}^{tr_+}_{1,\mathcal{H},0} \right)  \quad \text{for all } n \in \N.
	\end{equation*}
	Hence $\norm{u^+_n}^{p_-}_{1,\mathcal{H},0} $ must be bounded as by \eqref{Asf:CeramiAssumption} it holds
	\begin{equation*}
		tr_+
		= \frac{p^*_-(r_+ - l_-)}{p^*_- - l_-}
		= \frac{N p_-(r_+ - l_-)}{Np_- - N l_- + p_- l_-}
		< \frac{N p_-(r_+ - l_-)}{Np_- - N l_- + p_- (r_+ - p_-) \frac{N}{p_-}} = p_-.
	\end{equation*}
	This finishes the proof of the Claim.

	From \eqref{Eq:NegPartConvegence} and the Claim we know that $\{u_n\}_{n \in \N}$ is bounded in $\WHzero$, therefore there exists $u \in \WHzero$ and a subsequence $\{u_{n_k}\}_{k \in \N}$ such that
	\begin{equation*}
		u_{n_k} \weak u \quad \text{in } \WHzero.
	\end{equation*}
	Furthermore, by testing \eqref{Eq:C-convTestedV} with $v = u_{n_k} - u \in \WHzero$, as $\norm{u_{n_k} - u}_{1,\mathcal{H},0} \leq \norm{u_{n_k}}_{1,\mathcal{H},0} + \norm{u}_{1,\mathcal{H},0}$ is upper bounded uniformly in $k$, it follows
	\begin{equation*}
		\lim\limits_{k \to \infty} \lan \ph_+ ' (u_{n_k}) , u_{n_k} - u \ran = 0,
	\end{equation*}
	and by the strong continuity of Lemma \ref{Le:Propf} \ref{Propf:ConvergenceFterm} (note that $f_+ (x,s) = f(x,s^+)$ still satisfies \eqref{Asf:Caratheodory} and \eqref{Asf:WellDef})
	\begin{equation*}
		\lim\limits_{k \to \infty} \left \lan I_{f_+} ' (u_{n_k}) , u_{n_k} - u \right\ran = 0.
	\end{equation*}
	In conclusion we obtain
	\begin{equation*}
		\lim\limits_{k \to \infty} \lan A(u_{n_k}) , u_{n_k} - u \ran = 0.
	\end{equation*}
	The (S$_+$)-property of the double phase operator $A$ from Proposition \ref{Prop:PropertiesDoublePhaseOp}
	implies that
	\begin{equation*}
		u_{n_k} \to u \quad \text{in } \WHzero.
	\end{equation*}
\end{proof}

Now we are finally in the position to apply the mountain pass theorem.
\begin{theorem}\label{Th:ConstantSignSolutions}
	Let \eqref{H1} be satisfied and $f$ fulfill \eqref{Asf:Caratheodory}, \eqref{Asf:WellDef}, \eqref{Asf:GrowthInfty}, \eqref{Asf:GrowthZero} and \eqref{Asf:CeramiAssumption}. Then there exist nontrivial weak solutions of problem \eqref{Eq:Problem} $u_0,v_0 \in \WHzero \cap \Lp{\infty} $ such that $u_0 \geq 0$ and $v_0 \leq 0$ a.\,e.\,in $\Omega$.
\end{theorem}
\begin{proof}
	We can apply Theorem \ref{Th:MPT} to the truncated energy functionals $\ph_\pm$ because of Propositions \ref{Prop:RingOfMountains}, \ref{Prop:PointBeyondMountains} and \ref{Prop:CeramiCondition}. Then we know that there exist $u_0,v_0 \in \WHzero$ such that $\ph_+'(u_0) = 0 = \ph_-'(v_0)$ and
	\begin{equation*}
		\ph_+(u_0), \ph_-(v_0) \geq \inf_{\norm{u}_{1,\mathcal{H},0} = \delta} \ph_\pm(u) > 0 = \ph_+ (0).
	\end{equation*}
	This shows $u_0 \neq 0 \neq v_0$.

	Finally, observe that if we test $\ph_+'(u_0) = 0$ with $-u_0^-$ we obtain $\rho_{\mathcal{H}} (- \nabla u_0^-) = 0$, which by Proposition \ref{Prop:PropertiesModularDoublePhase} (i) and Proposition \ref{Prop:PoincareIneq} implies that $-u_0^- = 0$ a.$\;$e. in $\Omega$, so $u_0 = u_0^+ \geq 0$ a.\,e.\,in $\Omega$. With an analogous argument, $v_0 \leq 0$ a.\,e.\,in $\Omega$. The boundedness of the solutions follows from Theorem \ref{Th:BoundedSolutions}.
\end{proof}

%********************************************************************
\section{Sign-changing solution}\label{sign-changing-solutions}
%********************************************************************

In this section we will prove the existence of a solution with non-trivial positive and negative part on top of the two solutions from the previous section. This will be carried out by using the so-called Nehari manifold and we base our arguments on the ideas of the paper by  Gasi\'{n}ski-Winkert \cite{Gasinski-Winkert-2021}. For a broader description of the Nehari manifold method, check the book chapter by Szulkin-Weth \cite{Szulkin-Weth-2010}.

First, note that the Nehari manifold of $\ph$ is the set
\begin{equation*}
	\mathcal{N} = \left\{ u \in \WHzero \,:\, \lan \ph'(u) , u \ran = 0, \; u \neq 0 \right\}.
\end{equation*}
As the weak solutions of \eqref{Eq:Problem} are exactly the critical points of $\ph$, $\mathcal{N}$ still contains all the weak solutions of \eqref{Eq:Problem} except for possibly $u=0$.
Also note that this set might not be in general a manifold (and this will not be important in the present work), but in any case the usual name in the literature is Nehari manifold.

As we want to deal with sign-changing solutions, we will be actually more interested in the set
\begin{equation*}
	\mathcal{N}_0 = \left\{ u \in \WHzero \,:\, \pm u^\pm \in \mathcal{N} \right\}.
\end{equation*}

First, it is necessary to prove some interesting properties of $\mathcal{N}$ that will be useful later.

\begin{proposition}\label{Prop:NehariManifoldProps}
	Let \eqref{H1} be satisfied  and $f$ fulfill \eqref{Asf:Caratheodory}, \eqref{Asf:WellDef}, \eqref{Asf:GrowthInfty}, \eqref{Asf:GrowthZero} and \eqref{Asf:QuotientMono}. Then for any $u \in \WHzero\setminus\{0\}$ there exists a unique $t_u > 0$ such that $t_u u \in \mathcal{N}$. Furthermore, $\ph(t_u u) > 0$, $\frac{d}{dt} \ph(tu) > 0$ for $0 < t < t_u$, $\frac{d}{dt} \ph(tu) = 0$ for $t = t_u$, $\frac{d}{dt} \ph(tu) < 0$ for $t > t_u$, and therefore $\ph(tu) < \ph(t_u u)$ for all $0 < t \neq t_u$.
\end{proposition}

\begin{proof}
	Fix $u \in \WHzero\setminus\{0\}$ and let $k_u \colon [0 , \infty) \to \R$ be defined by $k_u(t) = \ph (tu)$. From the composition of functions, $k_u$ is $C^1$ in $(0,\infty)$ and continuous in $[0,\infty)$.
	By Propositions \ref{Prop:RingOfMountains} and \ref{Prop:PointBeyondMountains} one gets that there exist $K, \delta > 0$ such that
	\begin{equation}\label{Eq:SignPhi}
		k_u (t) > 0\quad \text{for } 0 < t < \delta
		\quad \text{ and } \quad
		k_u(t) < 0\quad \text{for } t > K,
	\end{equation}
	and clearly $k_u(0) = 0$. By this and the extreme value theorem, there exists $0 < t_u \leq K$ such that
	\begin{equation*}
		\sup_{t \in [0,\infty)} k_u(t)
		= \max_{t \in [0 , K]} k_u(t)
		= k_u (t_u).
	\end{equation*}
	In particular, $t_u$ is a local maximum in the interior of $[0,\infty)$, hence a critical point of $k_u$ and by the chain rule we have
	\begin{equation*}
		0 = k_u ' (t_u)
		= \lan \ph'(t_u u) , u \ran.
	\end{equation*}
	Thus, $t_u u \in \mathcal{N}$.

	It is now left to see the uniqueness of $t_u$, the sign of the derivatives and that it is a strict maximum. First note that \eqref{Asf:QuotientMono} for $t > 0$ implies (as functions only depending on $t$)
	\begin{align*}
		\frac{f(x,tu)}{t^{q_+-1}\abs{u}^{q_+-1}} \text{ increasing, hence } \frac{f(x,tu)u}{t^{q_+-1}} \text{ increasing } & \text{, for } x \in \Omega \text{ with } u(x) > 0,  \\
		\frac{f(x,tu)}{t^{q_+-1}\abs{u}^{q_+-1}} \text{ decreasing, hence } \frac{f(x,tu)u}{t^{q_+-1}} \text{ increasing } & \text{, for } x \in \Omega \text{ with } u(x) < 0.
	\end{align*}
	The equation $k_u'(t) = 0$ with $t>0$ is a necessary condition for $tu \in \mathcal{N}$. Multiplying this condition by $1/t^{q_+ - 1}$ one gets
	\begin{equation*}
		\into \left( \frac{ 1 } {t^{q_+ - p(x)}} | \nabla u |^{p(x)}
		+ \frac{1}{ t^{q_+ - q(x)} } \mu(x) | \nabla u |^{q(x)}
		- \frac{f(x,tu)u}{t^{q_+ - 1}} \right) \dx = 0,
	\end{equation*}
	where, by the comment above and $p(x) < q(x) \leq q_+$ for all $x \in \close$, the integrand is strictly decreasing in the set $\{ \nabla u \neq 0\}$ and at least decreasing outside. Hence the integral is strictly decreasing as a function of $t$, so there can be at most a single value for which the equation holds, i.e.\,there is at most a single value $t_u \in (0, \infty)$ such that $k_u'(t_u) = 0$. As a consequence there can be at most a single value $t_u \in (0, \infty)$ such that $t_u u \in \mathcal{N}$. On the other hand, $k_u'(t)$ is non-vanishing with constant sign in $(0,t_u)$ and in $(t_u,\infty)$, which by \eqref{Eq:SignPhi} must be positive and negative, respectively. Thus $t_u$ is a strict maximum of $k_u$.
\end{proof}

Another reason to restrict ourselves to $\mathcal{N}$ is that the restriction of our energy functional $\ph$ has better properties.

\begin{proposition}\label{Prop:Coercivity}
	Let \eqref{H1} be satisfied  and $f$ fulfill \eqref{Asf:Caratheodory}, \eqref{Asf:WellDef}, \eqref{Asf:GrowthInfty}, \eqref{Asf:GrowthZero} and \eqref{Asf:QuotientMono} . Then the functional $\ph_{|_\mathcal{N}}$ is sequentially coercive, in the sense that for any sequence $\{u_n\}_{n \in \N} \subseteq \mathcal{N}$ such that $\norm{u_n}_{1,\mathcal{H},0} \xrightarrow{n \to \infty} + \infty$ it follows that $\ph(u_n) \xrightarrow{n \to \infty} + \infty$.
\end{proposition}

\begin{proof}
	Let $\{u_n\}_{n \in \N} \subseteq \mathcal{N}$ be a sequence such that $\norm{u_n}_{1,\mathcal{H},0} \xrightarrow{n \to \infty} + \infty$ and  let $y_n = \frac{u_n}{\norm{u_n}_{1,\mathcal{H},0}}$ for all $n \in \N$. Then, by the compact embedding $\WHzero \hookrightarrow \Lp{r}$ of Proposition \ref{Prop:EmbeddingsH} (iii), there exist a subsequence $\{y_{n_k}\}_{k \in \N}$ and $y \in \WHzero$ such that
	\begin{equation*}
		y_{n_k} \weak y \quad\text{in } \WHzero,
		\quad y_{n_k} \to y \quad\text{in } \Lp{r(\cdot)}  \text{ and pointwisely a.\,e.\,in }\Omega.
	\end{equation*}

	{\bf Claim:} $y = 0$

	Assume that $y \neq 0$. By Proposition \ref{Prop:PropertiesModularDoublePhase} (iv) and for $k \geq k_0$ such that $\norm{u_{n_k}}_{1,\mathcal{H},0} \geq 1$ we obtain
	\begin{align*}
		\ph (u_{n_k}) & \leq \frac{1}{p_-} \varrho_{p(\cdot)} ( \nabla  u_{n_k} ) +
		\frac{1}{q_-} \varrho_{q(\cdot),\mu} ( \nabla u_{n_k} ) - \into F(x, u_{n_k}) \dx \\
		& \leq \frac{1}{p_-} \norm{u_{n_k}}_{1,\mathcal{H},0}^{q_+} - \into F(x, u_{n_k}) \dx.
	\end{align*}
	Therefore, dividing by $\norm{u_{n_k}}_{1,\mathcal{H},0}^{q_+}$, one gets
	\begin{equation*}
		\frac{\ph (u_{n_k})}{\norm{u_{n_k}}_{1,\mathcal{H},0}^{q_+}}
		\leq \frac{1}{p_-} - \into \frac{F(x, u_{n_k})}{\abs{u_{n_k}}^{q_+}} \abs{y_{n_k}}^{q_+} \dx
		\to - \infty\quad\text{as }k \to \infty,
	\end{equation*}
	where this limit follows from \eqref{Asf:GrowthInfty}, because
	\begin{equation*}
		\lim\limits_{k \to \infty} \frac{F(x, u_{n_k})}{\abs{u_{n_k}}^{q_+}} \abs{y_{n_k}}^{q_+} = + \infty \quad \text{for all } x \in \Omega \text{ with } y(x) \neq 0,
	\end{equation*}
	and also by Fatou's Lemma and Lemma \ref{Le:Propf} \ref{Propf:Boundedbelow}, taking $\Omega_0 = \{ x \in \Omega : y(x) = 0 \}$ in
	\begin{align*}
		\into \frac{F(x, u_{n_k})}{\abs{u_{n_k}}^{q_+}} \abs{y_{n_k}}^{q_+} \dx
		 & = \int_{\Omega \setminus \Omega_0} \frac{F(x, u_{n_k})}{\abs{u_{n_k}}^{q_+}} \abs{y_{n_k}}^{q_+} \dx + \int_{\Omega_0} \frac{F(x, u_{n_k})}{\norm{u_{n_k}}_{1,\mathcal{H},0}^{q_+}} \dx \\
		 & \geq \int_{\Omega \setminus \Omega_0} \frac{F(x, u_{n_k})}{\abs{u_{n_k}}^{q_+}} \abs{y_{n_k}}^{q_+} \dx - \frac{M |\Omega| }{\norm{u_{n_k}}_{1,\mathcal{H},0}^{q_+}}
		\xrightarrow{k \to \infty} + \infty.
	\end{align*}
	However, by Proposition \ref{Prop:NehariManifoldProps} we know that $\ph (u_n) > 0$ for all $n \in \N$, so the limit above yields a contradiction. This finishes the proof of the Claim.

	Now fix any $t > 1$. As $u_{n_k} \in \mathcal{N}$,  Proposition \ref{Prop:PropertiesModularDoublePhase} (iv) and Proposition \ref{Prop:NehariManifoldProps} yield for all $k \in \N$
	\begin{align*}
		\ph (u_{n_k}) & \geq \ph (t y_{n_k})
		\geq \frac{1}{p_+} \varrho_{p(\cdot)} ( \nabla (t y_{n_k}) ) +
		\frac{1}{q_+} \varrho_{q(\cdot),\mu} ( \nabla (t y_{n_k}) ) - \into F(x, ty_{n_k})                   \\
		              & \geq \frac{1}{q_+} \norm{t y_{n_k}}_{1,\mathcal{H},0}^{p_-} - \into F(x, ty_{n_k})
		= \frac{t^{p_-}}{q_+} - \into F(x, ty_{n_k}) \dx.
	\end{align*}
	Furthermore, as $t y_{n_k} \weak 0$, by the strong continuity of the $F$ term (see Lemma \ref{Le:Propf} \ref{Propf:ConvergenceFterm}) there exists $k_0 \in \N$ such that for $k \geq k_0$
	\begin{equation*}
		\ph (u_{n_k}) \geq \frac{t^{p_-}}{q_+} - 1.
	\end{equation*}
	As the choice of $t>1$ was arbitrary, $\ph (u_{n_k}) \xrightarrow{k \to \infty} + \infty$. The subsequence principle yields the result for the whole sequence, i.e.\,$\ph (u_n) \xrightarrow{n \to \infty} + \infty$.
\end{proof}

Now we will make use of the properties proved in the previous results to obtain the existence of a minimizer of our energy functional $\ph$ restricted to $\mathcal{N}_0$.

\begin{proposition}%\label{Prop:Infimum>0}
	Let \eqref{H1} be satisfied and $f$ fulfill \eqref{Asf:Caratheodory}, \eqref{Asf:WellDef}, \eqref{Asf:GrowthInfty},  \eqref{Asf:GrowthZero} and \eqref{Asf:QuotientMono}. Then
	\begin{equation*}
		\inf_{u \in \mathcal{N}} \ph(u) > 0
		\quad \text{ and } \quad
		\inf_{u \in \mathcal{N}_0} \ph(u) > 0.
	\end{equation*}
\end{proposition}

\begin{proof}
	Let $u \in \mathcal{N}$, by Propositions \ref{Prop:RingOfMountains} and \ref{Prop:NehariManifoldProps} it follows
	\begin{equation*}
		\ph (u)
		\geq \ph \left( \frac{\delta}{\norm{u}_{1,\mathcal{H},0}} u \right)
		\geq \inf_{\norm{u}_{1,\mathcal{H},0} = \delta} \ph(u)
		> 0.
	\end{equation*}
	Hence
	\begin{equation*}
		\inf_{u \in \mathcal{N}} \ph(u)
		\geq \inf_{\norm{u}_{1,\mathcal{H},0} = \delta} \ph(u)
		> 0.
	\end{equation*}
	Also, as for each $u \in \mathcal{N}_0$ it holds that $u^+, - u ^- \in \mathcal{N}$, it follows that
	\begin{equation*}
		\ph (u)
		= \ph(u^+) + \ph(-u^-)
		\geq 2 \inf_{u \in \mathcal{N}} \ph(u)
		> 0.
	\end{equation*}
	Therefore
	\begin{equation*}
		\inf_{u \in \mathcal{N}_0} \ph(u)
		\geq 2 \inf_{u \in \mathcal{N}} \ph(u)
		> 0.
	\end{equation*}
\end{proof}

\begin{proposition}\label{Prop:MinimizerExistence}
	Let \eqref{H1} be satisfied and $f$ fulfill \eqref{Asf:Caratheodory}, \eqref{Asf:WellDef}, \eqref{Asf:GrowthInfty},  \eqref{Asf:GrowthZero} and \eqref{Asf:QuotientMono}. Then there exists $w_0 \in \mathcal{N}_0$ such that $\ph(w_0) = \inf_{u \in \mathcal{N}_0} \ph(u)$.
\end{proposition}

\begin{proof}
	We will proceed by the direct method of calculus of variations. Let $\{u_n\}_{n \in \N} \subseteq \mathcal{N}_0$ be a minimizing sequence, that is,
	\begin{equation*}
		\lim\limits_{n \to \infty} \ph(u_n) = \inf_{u \in \mathcal{N}_0} \ph(u).
	\end{equation*}
	In the following, recall that for $v \in \WHzero$, we have $v^\pm \in \WHzero$ by Proposition \ref{Prop:ClosedTruncations}. As $\ph(u_n) = \ph(u_n^+) \mathop{+} \ph(-u_n^-)$ for all $n \in \N$, by Proposition \ref{Prop:NehariManifoldProps} it holds that $\ph(\pm u_n^\pm) > 0$ for all $n \in \N$, and by the coercivity of Proposition \ref{Prop:Coercivity} it follows that $\{\pm u_n ^\pm \}_{n \in \N}$ are both bounded sequences in $\mathcal{N}$. By the compact embedding $\WHzero \hookrightarrow \Lp{r}$ of Proposition \ref{Prop:EmbeddingsH} (iii), there exist subsequences $\{\pm u_{n_k} ^\pm\}_{k \in \N}$ and $u_\pm \in \WHzero$ such that
	\begin{equation}\label{Eq:MinimizingSubsequence}
		\pm u_{n_k} ^\pm \weak u_\pm \text{ in } \WHzero, \; \pm u_{n_k} ^\pm \to u_\pm \text{ in } \Lp{r(\cdot)} \text{ and pointwisely a.$\,$e. with}
		\begin{cases}
			u_{+} \geq 0, \\ u_{-} \leq 0, \\ u_+ u_- = 0.
		\end{cases}
		\mkern-25mu
	\end{equation}
	Assume that $u_\pm = 0$. Then, as $\pm u_{n_k} ^\pm \in \mathcal{N}$, we get
	\begin{equation*}
		0 = \lan \ph'(\pm u_{n_k} ^\pm) , \pm u_{n_k} ^\pm  \ran
		= \rho_{\mathcal{H}}( \pm \nabla u_{n_k} ^\pm) - \into f(x, \pm u_{n_k}^\pm) ( \pm u_{n_k} ^\pm ) \dx,
	\end{equation*}
	and by the strong continuity of the $f$ term (see Lemma \ref{Le:Propf} \ref{Propf:ConvergenceFterm}), it follows that $\rho_{\mathcal{H}}( \pm \nabla u_{n_k} ^\pm)$ $\to 0$. From Proposition \ref{Prop:PropertiesModularDoublePhase} (v) we derive that $\pm u_{n_k} ^\pm \to 0$ in $\WHzero$, so
	\begin{equation*}
		0 < \inf_{u \in \mathcal{N}} \ph(u)
		\leq \ph (\pm u_{n_k} ^\pm)
		\xrightarrow{k \to \infty} \ph (0) = 0,
	\end{equation*}
	which is a contradiction. Hence $u_+ \neq 0 \neq u_-$.

	By Proposition \ref{Prop:NehariManifoldProps}, there exist $t_\pm > 0$ such that $t_\pm u_\pm \in \mathcal{N}$. Take $w_0 = t_+ u_+ + t_- u_-$, by \eqref{Eq:MinimizingSubsequence} it holds that $\pm (w_0)^\pm = t_\pm u_\pm$ and hence $w_0 \in \mathcal{N}_0$.

	Finally, observe that $\ph$ is sequentially weakly lower semicontinuous because the mapping $I \colon \WHzero \to \R$ given by
	\begin{equation*}
		I(u) = \into  \frac{\abs{\nabla u}^{p(x)}}{p(x)} \dx \;
		+ \into \mu(x) \frac{\abs{ \nabla u}^{q(x)}}{q(x)}  \dx
	\end{equation*}
	is sequentially weakly lower semicontinuous (note that each addend is convex and continuous, hence each of them is sequentially weakly lower semicontinuous)
	%GRADIENT TERM IS NOT INCLUDED HERE (see the book by Diening-Harjulehto-H\"{a}st\"{o}-R\r{u}\v{z}i\v{c}ka \cite[Theorem 2.2.8]{DieningEtAl2011})
	and the $F$ term is strongly continuous by Lemma \ref{Le:Propf} \ref{Propf:ConvergenceFterm}. As a consequence, together with Proposition \ref{Prop:NehariManifoldProps}, this gives
	\begin{align*}
		\inf_{u \in \mathcal{N}_0} \ph(u)
		 & = \lim\limits_{k \to \infty} \ph(u_{n_k})
		= \lim\limits_{k \to \infty} \ph(u_{n_k}^+) + \ph(- u_{n_k} ^-)                  \\
		 & \geq \liminf\limits_{k \to \infty} \ph(t_+ u_{n_k}^+) + \ph(- t_- u_{n_k} ^-) \\
		 & \geq \ph(t_+ u_+) + \ph(t_- u_-)
		= \ph(w_0)
		\geq \inf_{u \in \mathcal{N}_0} \ph(u).
	\end{align*}
	Hence, $\ph(w_0)
		= \inf_{u \in \mathcal{N}_0} \ph(u)$.
\end{proof}

Finally, the minimizer obtained in Proposition \ref{Prop:MinimizerExistence} turns out to be a critical point of $\ph$ as well and so it is a weak solution of problem \eqref{Eq:Problem}.

\begin{proposition}\label{Prop:Sing-changing-CritPoint}
	Let \eqref{H1} be satisfied and $f$ fulfill \eqref{Asf:Caratheodory}, \eqref{Asf:WellDef}, \eqref{Asf:GrowthInfty},  \eqref{Asf:GrowthZero} and \eqref{Asf:QuotientMono}. Let $w_0 \in \mathcal{N}_0$ such that $\ph(w_0) = \inf_{u \in \mathcal{N}_0} \ph(u)$. Then $w_0$ is a critical point of $\ph$.
\end{proposition}

\begin{proof}
	First, two key observations for the proof.
	\begin{enumerate} [label=(\alph*)]
		\item
			\label{Eq:SmallerThanInfimum}
			By Proposition \ref{Prop:NehariManifoldProps}, for any $(s,t) \in [0,\infty)^2 \setminus \{(1,1)\}$
			\begin{equation*}
				\ph(s w_0^+ - t w_0^-)
				= \ph(s w_0^+) + \ph(- t w_0^-)
				< \ph(w_0^+) + \ph(- w_0^-)
				= \ph(w_0)
				= \inf_{u \in \mathcal{N}_0} \ph(u).
			\end{equation*}
		\item
			\label{Eq:BallSing-changing}
			For any $v \in \WHzero$, as $w_0^+ \neq 0 \neq w_0^-$
			\begin{equation*}
				\norm{ w_0 - v }_{1,\mathcal{H},0}
				\geq C_{\mathcal{H}}^{-1} \norm{ w_0 - v}_{p_-}
				\geq
				\begin{cases}
					C_{\mathcal{H}}^{-1} \norm{ w_0^- }_{p_-} & , \text{ if } v^- = 0, \\ C_{\mathcal{H}}^{-1} \norm{ w_0^+ }_{p_-} & , \text{ if } v^+ = 0,
				\end{cases}
			\end{equation*}
			where $C_\mathcal{H}>0$ is the constant of the embedding $\Wpzero{\mathcal{H}} \hookrightarrow \Lp{p_-}$. Hence, if for
			\begin{equation*}
				0 < \delta_0 < \min \left \{ C_{\mathcal{H}}^{-1} \norm{ w_0^+}_{p_-} , C_{\mathcal{H}}^{-1} \norm{ w_0^- }_{p_-} \right\},
			\end{equation*}
			%$0 < \delta_0 < \min\{ C_p^p \norm{ \nabla u^+ }_p^p , C_p^p \norm{ \nabla u^- }_p^p , 1 \}$
			we have $\norm{ w_0 - v }_{1,\mathcal{H},0} < \delta_0$,
			it follows that $v^+ \neq 0 \neq v^-$.
	\end{enumerate}

	The proof consists of arguing by contradiction using the deformation lemma (see Lemma \ref{Le:DeformationLemma}). Assume that $\ph'(w_0) \neq 0$. Then there exist $\alpha, \delta_1 > 0$ such that
	\begin{equation}
		\label{Eq:ContradictionNehari}
		\norm{\ph'(u)}_* \geq \alpha \quad\text{for all } u \in \WHzero \text{ with } \norm{ u - w_0 }_{1,\mathcal{H},0} < 3 \delta_1.
	\end{equation}

	Let
	\begin{equation*}
		\delta = \min \{ \delta_0 / 2, \delta_1 \}.
	\end{equation*}
	Consider again the mapping from \ref{Eq:SmallerThanInfimum} defined on $[0,\infty)^2 \to \WHzero$ and given by $(s,t) \mapsto s w_0^+ - t w_0^-$, which is continuous. Hence given $\delta > 0$, there exists $1 > \gamma > 0$ such that
	\begin{align}\label{Eq:gammachoice}
		\begin{split}
			&\norm{ s w_0^+ - t w_0^- - w_0}_{1,\mathcal{H},0} < \delta \text{ for all } (s,t) \in [0,\infty)^2\\
			&\text{ such that } \max\{ \abs{s-1} , \abs{t-1} \} < \gamma.
		\end{split}
	\end{align}

	Let
	\begin{equation*}
		D = (1 - \gamma, 1 + \gamma)^2.
	\end{equation*}
	By \ref{Eq:SmallerThanInfimum} it follows that
	\begin{equation*}
		\beta = \max_{(s,t) \in \partial D} \ph(s w_0^+ - t w_0^-)
		< \ph(w_0)
		= \inf_{u \in \mathcal{N}_0} \ph (u).
	\end{equation*}

	With the notation of the quantitative deformation lemma (Lemma \ref{Le:DeformationLemma}), let
	\begin{equation*}
		S = B(w_0, \delta), \quad
		c = \inf_{u \in \mathcal{N}_0} \ph(u), \quad
		\eps = \min \left\{ \frac{ c - \beta}{4} , \frac{\alpha \delta}{8} \right\} \quad
		\text{and}\quad \delta \text{ be as above.}
	\end{equation*}
	From \eqref{Eq:ContradictionNehari} it follows that
	\begin{equation*}
		\norm{\ph'(u)}_* \geq \alpha \geq \frac{8 \eps}{\delta} \quad \text{for all } u \in S_{2 \delta} = B(w_0, 3 \delta),
	\end{equation*}
	so all the assumptions of the lemma are fulfilled. Furthermore
	\begin{align}
		\label{Eq:DeformBorderCond}
		\ph(s w_0^+ - t w_0^-)
		\leq \beta + c - c
		< c - \left( \frac{c - \beta}{2} \right)
		\leq c - 2 \eps \quad \text{for all } (s,t) \in \partial D.
	\end{align}

	Altogether, by the quantitative deformation lemma there exists a mapping $\eta \in C ([0,1] \times \WHzero , \WHzero)$ such that
	\begin{enumerate} [label=(\roman*),font=\normalfont]
		\item
			$\eta (t, u) = u$, if $t = 0$ or if $u \notin \ph^{-1}([c - 2\eps, c + 2\eps]) \cap S_{2 \delta}$,
		\item
			$\ph( \eta( 1, u ) ) \leq c - \eps$ for all $u \in \ph^{-1} ( ( - \infty, c + \eps] ) \cap S $,
		\item
			$\eta(t, \cdot )$ is an homeomorphism of $\WHzero$ for all $t \in [0,1]$,
		\item
			$\norm{\eta(t, u) - u} \leq \delta$ for all $u \in \WHzero$ and $t \in [0,1]$,
		\item
			$\ph( \eta( \cdot , u))$ is decreasing for all $u \in \WHzero$,
		\item
			$\ph(\eta(t, u)) < c$ for all $u \in \ph^{-1} ( ( - \infty, c] ) \cap S_\delta$ and $t \in (0, 1]$.
	\end{enumerate}

	We define now $h \colon [0,\infty)^2 \to \WHzero$ given by $h(s,t) = \eta ( 1 , s w_0^+ - t w_0^-)$, which has the following properties:
	\begin{enumerate} [label=(\roman*),font=\normalfont]
		\setcounter{enumi}{6}
		\item
			$h \in C \left( [0,\infty)^2 , \WHzero \right)$,
		\item
			$\ph( h(s,t) ) \leq c - \eps$ for all $(s,t) \in D$, by (ii), \ref{Eq:SmallerThanInfimum} and \eqref{Eq:gammachoice},
		\item
			$h(s,t) = s w_0^+ - t w_0^-$ for all $(s,t) \in \partial D$, by (i) and \eqref{Eq:DeformBorderCond},
		\item
			$h(D) \subseteq S_\delta = B(w_0, 2 \delta)$, by (iv) and \eqref{Eq:gammachoice}.
	\end{enumerate}

	Consider the mappings $H_0,H_1 \colon (0,\infty)^2 \to \R^2$ defined by
	\begin{align*}
		 & H_0 (s,t) = \left( \; \lan \ph'(s w_0^+) , w_0^+ \ran \; , \; \lan \ph'(- t w_0^-) , - w_0^- \ran \; \right),                                   \\
		 & H_1 (s,t) = \left( \; \frac{1}{s} \lan \ph'(h^+(s,t)) , h^+(s,t) \ran \; , \; \frac{1}{t} \lan \ph'(- h^- (s,t)) , -h^- (s,t) \ran \; \right).
	\end{align*}
	As $\ph$ is a $C^1$ functional and by (vii) both are continuous functions. By Proposition \ref{Prop:NehariManifoldProps} we have
	\begin{align*}
		 & \lan \ph'(t w_0^+) , w_0^+ \ran > 0 \quad \text{ and } \quad \lan \ph'(- t w_0^-) , - w_0^- \ran > 0 \quad \text{for all }0 < t < 1, \\
		 & \lan \ph'(t w_0^+) , w_0^+ \ran < 0 \quad \text{ and } \quad \lan \ph'(- t w_0^-) , - w_0^- \ran < 0 \quad \text{for all } t > 1.
	\end{align*}
	Let us denote by $\deg (g,A,y)$ the Brouwer degree over $A \subseteq \R^n$ open and bounded of the function $g \in C(A,\R^N)$ at the value $y \in \R^N \setminus g(\partial A)$. By its Cartesian product property (see Dinca-Mawhin \cite[Lemma 7.1.1 and Theorem 7.1.1]{Dinca-Mawhin-2021}) and the one-dimensional case over intervals (in the same book \cite[Proposition 1.2.3]{Dinca-Mawhin-2021}) we have
	\begin{align*}
		\deg (H_0, D, 0)
		 & = \deg \left( \lan \ph'(s w_0^+) , w_0^+ \ran , (1 - \gamma, 1 + \gamma) , 0 \right)
		\deg \left( \lan \ph'(- t w_0^-) , - w_0^- \ran , (1 - \gamma, 1 + \gamma) , 0 \right)    \\
		 & = (-1) (-1) = 1.
	\end{align*}
	By (ix) we have ${H_0} _{|_{\partial D}} = {H_1} _{|_{\partial D}}$, so by the dependence on the boundary values of the Brouwer degree (see again the book by Dinca-Mawhin \cite[Corollary 1.2.7]{Dinca-Mawhin-2021}), it follows
	\begin{equation*}
		\deg (H_1, D, 0)
		= \deg (H_0, D, 0)
		= 1,
	\end{equation*}
	which by the solution property of the Brouwer degree  (in the same book \cite[Corollary 1.2.5]{Dinca-Mawhin-2021}) implies that there exists $(s_0 , t_0) \in D$ such that $H_1 (s_0 , t_0) = (0,0)$, i.e.
	\begin{equation*}
		\lan \ph'(h^+(s_0 , t_0)) , h^+(s_0 , t_0) \ran
		= 0
		= \lan \ph'(- h^- (s_0 , t_0)) , -h^- (s_0 , t_0) \ran.
	\end{equation*}
	Furthermore, by (x)
	\begin{equation*}
		\norm{h(s_0 , t_0) - w_0}_{1,\mathcal{H},0} \leq 2 \delta \leq \delta_0,
	\end{equation*}
	and then \ref{Eq:BallSing-changing} yields
	\begin{equation*}
		h^+(s_0 , t_0) \neq 0 \neq - h^-(s_0 , t_0).
	\end{equation*}
	Altogether, we have that $h(s_0 , t_0) \in \mathcal{N}_0$, which by (viii) fulfills $\ph(h(s_0 , t_0)) < c - \eps$. This is a contradiction with $c = \inf_{u \in \mathcal{N}_0} \ph(u)$ and the proof is complete.
\end{proof}

Now we can combine Theorems \ref{Th:BoundedSolutions}, \ref{Th:ConstantSignSolutions} and Propositions \ref{Prop:MinimizerExistence}, \ref{Prop:Sing-changing-CritPoint} to get the following result.
\begin{theorem}\label{Th:ThreeSolutions}
	Let \eqref{H1} be satisfied and $f$ fulfill  \eqref{Asf:Caratheodory}, \eqref{Asf:WellDef}, \eqref{Asf:GrowthInfty},  \eqref{Asf:GrowthZero}, \eqref{Asf:CeramiAssumption} and \eqref{Asf:QuotientMono}. Then there exist nontrivial weak solutions of problem \eqref{Eq:Problem} $u_0,v_0,w_0 \in \WHzero \cap \Lp{\infty}$ such that $u_0 \geq 0$, $v_0 \leq 0$ and $w_0$ has changing sign.
\end{theorem}

%********************************************************************
\section{Nodal domains} \label{nodal-domains}
%********************************************************************

It is possible to derive further properties of the sign-changing solution obtained in the previous section. We will deal here with the number of nodal domains, that is, the number of maximal regions where it does not change sign. The usual definition of nodal domains for $u \in C(\Omega,\R)$ are the connected components of $\Omega \setminus Z$, where the set $Z = \{ x \in \Omega \,:\, u(x) = 0 \}$ is called the nodal set of $u$. However, since we do not know whether our solutions are continuous, this is not helpful. Thus we propose the following, alternative definition.

Let $u \in \WHzero$ and $A$ be a Borelian subset of $\Omega$ with $|A| > 0$. We say $A$ is a nodal domain of $u$ if
\begin{enumerate}[label=(\roman*),font=\normalfont]
	\item
		$u \geq 0$ a.\,e.\,on $A$ or $u \leq 0$ a.\,e.\,on $A$;
	\item
		$0 \neq u 1_A \in \WHzero$;
	\item
		$A$ is minimal w.r.t.\,\textnormal{(i)} and \textnormal{(ii)}, i.e., if $B \subseteq A$ with $B$ being a Borelian subset of $\Omega$, $|B| > 0$ and $B$ satisfies \textnormal{(i)} and \textnormal{(ii)}, then $|A \setminus B| = 0$.
\end{enumerate}

\begin{remark}$~$
	\begin{enumerate}
		\item[\textnormal{(a)}]
			The reason to assume (ii) in the previous definition is that we want to rule out vanishing sets and sets such that the weak derivative of $u 1_A$ does not exist.
		\item[\textnormal{(b)}]
			The relation in general between the usual definition and this one is unknown for the authors.
	\end{enumerate}
\end{remark}

In order to prove the extra properties of our sign-changing solution we also need the assumption \eqref{Asf:IntMono}.

\begin{proposition}\label{Prop:NodalDomains}
	Let \eqref{H1} be satisfied and $f$ fulfill \eqref{Hf}. Then, any minimizer of $\ph_{|_{\mathcal{N}_0}}$ (which by Proposition \ref{Prop:Sing-changing-CritPoint} is also a sign-changing weak solution of \eqref{Eq:Problem}) has exactly two nodal domains.
\end{proposition}

\begin{proof}
	We fix any representative of $w_0$ denoted by $\widetilde{w_0}$ and set $\Omega_\pm = \{ x \in \Omega : \pm \widetilde{w_0} (x) > 0 \}$. Both, $\Omega_+$ and $\Omega_-$, satisfy conditions \textnormal{(i)} and \textnormal{(ii)} in the definition above because $w_0 1_{\Omega_\pm} = \pm \widetilde{w_0} ^\pm$ a.\,e.\,in $\Omega$.

	It remains to show that they are minimal. The proof follows arguing by contradiction. Without loss of generality, assume that there exist Borelian subsets  $A_1, A_2$ of $\Omega$ such that they are disjoint, of positive measure, with $\Omega_- = A_1 \dot{\cup} A_2$ and $A_1$ satisfies \textnormal{(i)} and \textnormal{(ii)} in the definition above. As a consequence, $A_2$ also satisfies \textnormal{(i)} and \textnormal{(ii)} because $w_0 1_{A_2} = w_0 1_{\Omega_-} - w_0 1_{A_1} \in \WHzero$ and $w_0 1_{A_2} = \widetilde{w_0} 1_{A_2} < 0$ a.\,e.\,in $A_2$. Hence we have
	\begin{equation*}
		1_{\Omega_+} w_0 \geq 0, \quad 1_{A_1} w_0 \leq 0, \quad 1_{A_2} w_0 \leq 0 \quad \text{ for a.\,a.\,} x \in \Omega
	\end{equation*}
	and
	\begin{equation*}
		1_{\Omega_+} w_0 + 1_{A_1} w_0 + 1_{A_2} w_0
			= 1_{\Omega_+} w_0 + 1_{\Omega_-} w_0 + 1_{\Omega_0} 0
		= w_0 \quad \text{ for a.\,a.\,} x \in \Omega,
	\end{equation*}
	where $\Omega_0 = \{ x \in \Omega : \widetilde{w_0} (x) = 0 \}$.

	Take $u = 1_{\Omega_+} w_0 + 1_{A_1} w_0$ and $v = 1_{A_2} w_0$ and note that $u ^+= 1_{\Omega_+} w_0$ and $- u ^- = 1_{A_1} w_0$. As $\ph'(w_0) = 0$ and as the supports of $u^+ , - u^-$ and $v$ are disjoint, it holds
	\begin{equation*}
		0 = \lan \ph'(w_0) , u^+ \ran
		= \lan \ph'(u^+) , u^+ \ran,
	\end{equation*}
	thus $u^+ \in \mathcal{N}$. Analogously, $- u^- \in \mathcal{N}$ and hence $u \in \mathcal{N}_0$. With the same argument it also holds that $\lan \ph'(v) , v \ran = 0$.

	Finally, by these properties and condition \eqref{Asf:IntMono}, we get
	\begin{align*}
		\inf_{u \in \mathcal{N}_0} \ph(u)
		 & = \ph (w_0)
		= \ph(u) + \ph(v) - \frac{1}{q_+} \lan \ph'(v) , v \ran                                                                                                  \\
		 & \geq \ph(u) + \left( \frac{1}{p_+} - \frac{1}{q_+} \right) \varrho_{p(\cdot)} ( \nabla v ) + \into \left( \frac{1}{q_+} f(x,v) v - F(x,v) \right) \dx \\
		 & \geq \ph(u) + \left( \frac{1}{p_+} - \frac{1}{q_+} \right) \varrho_{p(\cdot)} ( \nabla v )                                                            \\
		 & \geq \inf_{u \in \mathcal{N}_0} \ph(u) + \left( \frac{1}{p_+} - \frac{1}{q_+} \right) \varrho_{p(\cdot)} ( \nabla v ),
	\end{align*}
	which is a contradiction because $p_+ < q_+$ and $v \neq 0$. This finishes the proof.
\end{proof}

Finally, combining Theorem \ref{Th:ThreeSolutions} and Proposition \ref{Prop:NodalDomains}, we have the following result.

\begin{theorem}
	Let \eqref{H1} be satisfied and $f$ fulfill \eqref{Hf}. Then there exist nontrivial weak solutions of problem \eqref{Eq:Problem} $u_0,v_0,w_0 \in \WHzero \cap \Lp{\infty}$ such that
	\begin{equation*}
		u_0 \geq 0, \quad v_0 \leq 0, \quad w_0 \text{ being sing-changing with two nodal domains.}
	\end{equation*}
\end{theorem}

%%***********************************************************************************
%\section*{Acknowledgment}
%%***********************************************************************************
%The authors wish to thank the knowledgeable reviewer for her/his very careful reading of the manuscript and very helpful suggestions in order to improve the paper.

%********************************************************************
\subsection*{Funding information:}
%********************************************************************

\'{A}ngel Crespo-Blanco was funded by the Deutsche Forschungsgemeinschaft (DFG, German Research Foundation) under Germany's Excellence Strategy - 	The Berlin Mathematics Research Center MATH+ and the Berlin Mathematical School (BMS) (EXC-2046/1, project ID: 390685689).

%%********************************************************************
%\subsection*{Author contribution:}
%%********************************************************************
%Ángel Crespo-Blanco and Patrick Winkert wrote the main manuscript.
%All authors reviewed the manuscript.
%
%%********************************************************************
%\subsection*{Declarations of interest:}
%%********************************************************************
%none.
%
%%********************************************************************
%\subsection*{Data Availability Statement:}
%%********************************************************************
%data sharing is not applicable to this article as no new data were created or analyzed in this study.

%********************************************************************

\end{document}